\documentclass[11pt]{amsart}
\usepackage{amsmath,amsthm,amssymb,enumerate,mathscinet,mathtools}
\usepackage{fullpage}
\usepackage{graphicx,xcolor,pgfplots}
\usepackage{verbatim}
\usepackage{mathrsfs}
\usepackage{hyperref}
\usepackage{enumitem}

\usepackage{scalerel}
\usepackage{dsfont}
\usepackage{blindtext}
\usepackage{multicol}
\newtheorem{theorem}{Theorem}[section]
\newtheorem{fact}[theorem]{Fact}
\newtheorem{lemma}[theorem]{Lemma}

\newtheorem{claim}[theorem]{Claim}
\newtheorem{definition}[theorem]{Definition}

\newtheorem{prop}[theorem]{Proposition}

\numberwithin{equation}{section}

\def\eps{\varepsilon}
\newcommand{\ep}{\varepsilon}

\newcommand{\bi}{\leftrightarrow}

\newcommand\blfootnote[1]{%
  \begingroup
  \renewcommand\thefootnote{}\footnote{#1}%
  \addtocounter{footnote}{-1}%
  \endgroup
}

\hypersetup{
    colorlinks,
    linkcolor={red!80!black},
    citecolor={blue!50!black},
    urlcolor={blue!80!black}
}

\def\N{\mathbb{N}}

\def\Prob{\mathbb{P}}
\def\cH{\mathcal{H}}

\newcommand{\forw}[1]{\overrightarrow{#1}}
\newcommand{\back}[1]{\overleftarrow{#1}}

\makeatletter
\newcommand*{\rom}[1]{\expandafter\@slowromancap\romannumeral #1@}
\makeatother

\DeclarePairedDelimiter\floor{\lfloor}{\rfloor} 
\DeclarePairedDelimiter\ceil{\lceil}{\rceil} 

\def\COMMENT#1{}
\let\COMMENT=\footnote% COMMENT OUT for clean output
\usepackage{caption}

\tikzstyle{every node}=[circle, draw, fill=black!50, inner sep=0pt, minimum width=2pt]

\usetikzlibrary{patterns}
\usetikzlibrary{decorations.markings, decorations.pathreplacing}

\tikzset{middlearrow/.style={
  decoration={markings,
   mark= at position 0.5 with {\arrow[very thick]{#1}} ,
  },
  postaction={decorate}
 }
}

\allowdisplaybreaks

\title{On oriented  cycles in randomly perturbed digraphs}

\author{Igor Araujo \and J\'ozsef Balogh \and Robert A. Krueger \and Sim\'on Piga \and Andrew Treglown}

\thanks{IA: Department of Mathematics, University of Illinois at Urbana-Champaign, Urbana, Illinois 61801, USA. Email: \texttt{igoraa2@illinois.edu}. Research partially supported by UIUC Campus Research Board RB 22000.
\\ \indent JB: Department of Mathematics, University of Illinois at Urbana-Champaign, Urbana, Illinois 61801, USA. Email: \texttt{jobal@illinois.edu}. Research partially supported by NSF Grant DMS-1764123, Arnold O. Beckman Research Award (UIUC Campus Research Board RB 22000), the Langan Scholar Fund (UIUC), and NSF RTG Grant DMS-1937241.
\\ \indent RK: Department of Mathematics, University of Illinois at Urbana-Champaign, Urbana, Illinois 61801, USA. Email: \texttt{rak5@illinois.edu}. Research supported  the NSF Graduate Research Fellowship Program  Grant No.\ DGE 21-4675.
\\ \indent SP: University of Birmingham, United Kingdom. Email: \texttt{s.piga@bham.ac.uk}. Research supported by EPSRC grant EP/V002279/1.
\\ \indent AT: University of Birmingham, United Kingdom. Email: \texttt{a.c.treglown@bham.ac.uk}. Research supported by EPSRC grant EP/V002279/1.}
\begin{document}

\maketitle

\begin{abstract}
In 2003, Bohman, Frieze, and Martin  initiated the study of randomly perturbed graphs and digraphs. For digraphs, they showed that for every $\alpha>0$, there exists a constant $C$ such that for every   $n$-vertex digraph of minimum semi-degree at least $\alpha n$, if one adds $Cn$ random edges then asymptotically almost surely the resulting digraph contains a consistently oriented Hamilton cycle.
We generalize their result, showing that the hypothesis of this theorem actually asymptotically almost surely ensures the existence of every orientation of a cycle of every possible length, simultaneously. 
Moreover, we prove that we can relax the minimum semi-degree condition to a minimum total degree condition when
considering orientations of a cycle that do not contain a large number of vertices of indegree $1$.
Our proofs make use of a variant of an absorbing method of Montgomery. 
\end{abstract}

\blfootnote{The main results of this paper were first announced in the conference abstract~\cite{conf}.}

\section{Introduction}
Hamilton cycles are one of the most studied objects in graph theory, and several classical results measure how `dense’ a graph needs to be to force a Hamilton cycle. In particular, in 1952 Dirac~\cite{D} proved that every $n$-vertex graph with minimum degree $\delta(G) \geq n/2$ contains a Hamilton cycle; the minimum degree condition here is best possible.

The Hamiltonicity of directed graphs has also been extensively investigated since the 1960s. A \emph{directed graph}, or \emph{digraph}, is a set of vertices together with a set of ordered pairs of distinct vertices. We think of a digraph as a loop-free multigraph, where every edge is given an orientation from one endpoint to another, and there is at most one edge oriented in each of the two directions between a pair of vertices. An \emph{oriented graph} is a digraph with at most one directed edge between every pair of vertices. An edge from vertex $u$ to vertex $v$ is represented as $\forw{uv}$ or $\back{vu}$.
In the digraph setting, there is more than one natural analog of the minimum degree of a graph.
The \emph{minimum semi-degree $\delta^0(D)$} of a digraph $D$ is the minimum of all the in- and outdegrees of the vertices in $D$; the \emph{minimum total degree $\delta(D)$} is the minimum number of edges incident to a vertex in $D$.
Ghouila-Houri~\cite{GhouilaHouri} proved that every strongly connected $n$-vertex digraph $D$ with minimum total degree $\delta(D)\geq n$ contains a \emph{consistently oriented} Hamilton cycle, that is, a cycle $(v_1, v_2, \dots, v_n, v_{n+1}=v_1)$ with edges $\forw{v_i v_{i+1}}$ for all $i \in [n]$. Note that there are $n$-vertex digraphs $D$ with $\delta(D) = 3n/2 - 2$ that do not contain a consistently oriented Hamilton cycle, so the strongly connected condition in Ghouila-Houri's theorem is necessary. 

An immediate consequence of Ghouila-Houri's theorem is that having minimum semi-degree $\delta^0(D) \geq n/2$ forces a consistently oriented Hamilton cycle, and this is best possible. After earlier partial results~\cite{Grant, Hagg}, DeBiasio, K\"uhn, Molla, Osthus, and Taylor~\cite{dkmot} proved that this minimum semi-degree condition in fact forces all possible orientations of a Hamilton cycle, except for the \emph{anti-directed} Hamilton cycle, that is, a cycle $(v_1, v_2, \dots, v_n, v_{n+1} = v_1)$ with edges $\forw{v_i v_{i+1}}$ for all odd $i \in [n]$ and $\back{v_i v_{i+1}}$ for all even $i \in [n]$, where $n$ is even. Earlier, DeBiasio and Molla~\cite{deb} showed that the minimum semi-degree threshold for forcing the anti-directed Hamilton cycle is in fact $\delta^0(D) \geq n/2+1$. 

There has also been interest in Hamilton cycles in random digraphs: the \emph{binomial random digraph $D(n,p)$} is the digraph with vertex set $[n]$, where each of the $n(n-1)$ possible directed edges is present with probability $p$, independently of all other edges. Recently, Montgomery~\cite{mont} determined the sharp threshold for the appearance of any fixed orientation of a Hamilton cycle $H$ in $D(n,p)$, thereby answering a conjecture of Ferber and Long~\cite{fl} in a strong form. Depending on the orientation of $H$, the threshold here can vary from $p=\log n/2n$ to $p=\log n/n$. 

In this paper, we find arbitrary orientations of Hamilton cycles in the randomly perturbed digraph model.
Introduced in both the undirected and directed setting by Bohman, Frieze, and Martin~\cite{bfm1}, this model starts with a dense (di)graph and then adds $m$ random edges to it. The overarching question now is how many random edges are required to ensure that the resulting (di)graph \emph{asymptotically almost surely (a.a.s.)} satisfies a given property, that is, with probability tending to $1$ as the number of vertices $n$ tends to infinity.
For example, Bohman, Frieze, and Martin~\cite{bfm1} proved that for every $\alpha>0$, there is a $C=C(\alpha)$ such that if we start with an arbitrary $n$-vertex graph $G$ of minimum degree $\delta(G)\geq \alpha n$ and add $Cn$ random edges to it, then a.a.s.\ the resulting graph is Hamiltonian. Furthermore, given a constant $0<\alpha <1/2$, in a complete bipartite graph with part sizes $\alpha n$ and $(1-\alpha)n$, a linear number of random edges are needed to ensure Hamiltonicity. Thus their result is best possible up to the dependence of $C$ on $\alpha$.  Subsequently, there has been a significant effort to improve our understanding of randomly perturbed graphs. 
See, e.g.,~\cite[Section 1.3]{hmt} and the references within for a snapshot of some of the results in the area.

Bohman, Frieze, and Martin~\cite{bfm1} also proved the analogous result for consistently oriented Hamilton cycles in the randomly perturbed digraph model. Their result is also best possible up to the dependence of $C$ on~$\alpha$, for similar reasons as the undirected setting.
\begin{theorem}[Bohman, Frieze, and Martin~\cite{bfm1}]\label{bfmthm}
 For every $\alpha > 0$, there is a $C=C(\alpha)$ such that if $D_0$ is an $n$-vertex digraph of minimum semi-degree $\delta^0(D_0)\geq \alpha n$, then $D_0 \cup D(n,C/n)$ a.a.s.~contains a consistently oriented Hamilton cycle. 
\end{theorem}
A notion closely related to Hamiltonicity is \emph{pancyclicity}, which is when a (di)graph contains cycles of every possible length. Bondy~\cite{bondy2} generalized Dirac's theorem, showing that if $\delta(G) \geq n/2$ then $G$ is pancyclic or $K_{n/2,n/2}$. Shortly after, Bondy~\cite{bondy1} proposed his famous meta-conjecture that any `non-trivial' sufficient condition for Hamiltonicity should be a sufficient condition for pancyclicity, up to a small number of exceptional graphs.
Krivelevich, Kwan, and Sudakov~\cite{kks2} generalized Theorem~\ref{bfmthm} in this way, showing that the same conditions as in Theorem~\ref{bfmthm} imply that the randomly perturbed digraph contains consistently oriented cycles of every length.
\begin{theorem}[Krivelevich, Kwan, and Sudakov~\cite{kks2}]\label{kks2}
 For every $\alpha > 0$,  there is a $C=C(\alpha)$ such that if $D_0$ is an $n$-vertex digraph of minimum semi-degree $\delta^0(D_0)\geq \alpha n$, then $D_0 \cup D(n,C/n)$ a.a.s.~contains a consistently oriented cycle of every length between $2$ and $n$.
\end{theorem}

The original rotation-extension-type proofs of Theorems~\ref{bfmthm} and~\ref{kks2} only guarantee consistently oriented cycles. Our main result is a generalization of Theorem~\ref{kks2} to allow for all orientations of a cycle of every possible length.
Moreover, we find all these cycles simultaneously, i.e., $D_0 \cup D(n,C/n)$ a.a.s.~ contains all of them. This last property is an example of \emph{universality}, a notion both well-studied in the random graph (e.g.,~\cite{univ, mont})
and randomly perturbed (e.g.,~\cite{bhkmpp, par}) settings.

\begin{theorem} \label{thm::semideg_simple}
 For every $\alpha > 0$, there is a $C=C(\alpha)$ such that if $D_0$ is an $n$-vertex digraph of minimum semi-degree $\delta^0(D_0) \geq \alpha n$, then $D_0 \cup D(n,C/n)$ a.a.s.~contains every orientation of a cycle of every length between $2$ and $n$.
\end{theorem}
Theorem~\ref{thm::semideg_simple} is best possible in the sense that one really needs to add a linear number
of random edges to $D_0$. Indeed, similarly as before, let $D$ be  the complete bipartite digraph  with part sizes $\alpha n$ and $(1-\alpha)n$ (where $0<\alpha <1/2$). Then one needs to add a linear number of edges to $D$ to ensure a Hamilton cycle of \emph{any} orientation.

It is also natural to try and generalize Theorem~\ref{bfmthm} in another direction, by relaxing the minimum semi-degree condition to a total degree. Unfortunately, this cannot be true for a Hamilton cycle $H$ in which all but $o(n)$ vertices have in- and outdegree $1$. Indeed, given $0< \alpha <1/2$, let $D$ be the $n$-vertex digraph
which consists of vertex classes $S$ and $T$ of sizes $\alpha n$ and $(1-\alpha)n$ respectively, and whose edge set consists of all possible edges with their startpoint in $S$ and their endpoint in $T$. Then whilst $\delta(D)= \alpha n$, given any constant $C$, with probability bounded away from $0$, $D \cup D(n,C/n)$ contains a linear number of vertices
with outdegree $0$ and a linear number of vertices
with indegree $0$, so it will not contain $H$. 

On the other hand, we show that this type of orientation of a Hamilton cycle is the only one we cannot guarantee. That is, our desired relaxation is possible for all orientations of a Hamilton cycle that contain a linear number of vertices of in- or outdegree $2$.

\begin{theorem}\label{thm:main}
 For every $\alpha , \eta> 0$, there is a~$C=C(\alpha, \eta)$ such that if $D_0$ is an $n$-vertex digraph of minimum total degree $\delta(D_0) \geq 2\alpha n$, then $D_0 \cup D(n,C/n)$ a.a.s.~contains every orientation of a cycle of every length between $2$ and $n$ that contains at most $(1-\eta) n$ vertices of indegree $1$.
\end{theorem}

The proof of Theorem~\ref{thm:main} has the same core ideas as the proof of Theorem~\ref{thm::semideg_simple}, but there are additional complications and technicalities that come with the weakened degree condition.

\smallskip

{\bf \noindent Notation.}
Throughout this paper we omit floors and ceilings whenever this
does not affect the argument.
Given a digraph $D$ we write $V(D)$ and $E(D)$ for its vertex and edge sets respectively. 
Given some $X \subseteq V(D),$ we write $D[X]$ for the induced subdigraph of $D$ with vertex set $X$.
Given some $x \in V(D)$, $N_D^+(x)$ denotes the out-neighborhood of $x$ in $D$, which is the set of vertices $y\in V(D)$ for which $\forw{xy} \in E(D)$; the outdegree of $x$ in $D$ is denoted by $d_D^+(x) := |N_D^+(x)|$. We define $N_D^-(x)$ and $d_D^-(x)$ analogously, and often omit the subscript when the digraph $D$ considered is clear from the context.

We write $\overleftrightarrow{u v}$ if $\forw{uv}$ and 
$\back{uv}$ are edges and call $\overleftrightarrow{u v}$ a \emph{bidirected edge}.
A \emph{bidirected path} is a digraph obtained from an undirected path by replacing each 
edge $uv$ with a bidirected edge $\overleftrightarrow{u v}$. An \emph{oriented path} is a 
 digraph obtained from an undirected path by replacing each 
edge $uv$ with a single directed edge; either $\forw{uv}$ or $\back{uv}$.
Given an oriented or bidirected path $P=(u_1,\dots,u_k)$ we call $u_1$ its \emph{startpoint} and $u_k$ its \emph{endpoint}, distinguishing it from the path $(u_k, \dots, u_1)$.

Given an oriented path $P = (u_1, \dots, u_k)$, we define $\sigma(u_i u_{i+1})$ to be $+$ if $\forw{u_i u_{i+1}} \in E(P)$ and $-$ otherwise. Given any $i<j$, when clear from the context, we  write $(u_i,\dots, u_j)$ to mean the oriented  subpath of $P$ on vertices $u_i,\dots, u_j$; so crucially, the edges in $(u_i,\dots, u_j)$ are oriented precisely as in $P$.

Given two oriented paths $P = (u_1, \dots, u_k)$ and $P' = (u_1', \dots, u_{k'}')$ with $u_k = u_1'$ and $V(P) \cap V(P') = \{u_k\}$, the \emph{concatenation of $P$ and $P'$}, denoted $P \circ P'$, is the path $(u_1, \dots, u_k, u_2', u_3', \dots, u_{k'}')$.

\smallskip

The paper is organized as follows. In the next section we give an outline of the proof of Theorem~\ref{thm::semideg_simple}.
In Section~\ref{sec::ingredients} we collect together various properties of random and pseudorandom digraphs.
The main work of the paper is the proof of our absorbing lemmas, one for each of our two theorems, which are given in Section~\ref{sec::absorber} for Theorem~\ref{thm::semideg_simple} and Section~\ref{sec::total} for Theorem~\ref{thm:main}.
We prove Theorems~\ref{thm::semideg_simple} and~\ref{thm:main} in Sections~\ref{sec::proof} and~\ref{sec::totalproof}, respectively. In Section~\ref{sec::conc} we give some concluding remarks.

\section{Overview of the proof of Theorem~\ref{thm::semideg_simple}}\label{subsec::ideas}

Our goal is to show that for a given orientation $\mathcal{C}$ of a cycle, $D_0 \cup D(n,C/n)$ contains $\mathcal{C}$ with probability at least $1-e^{-n}$. Theorem~\ref{thm::semideg_simple} follows from a union bound over all choices of $\mathcal{C}$, of which there are trivially at most $n2^n$. For the rest of this section we consider only spanning $\mathcal{C}$, as the non-spanning cycle case follows easily from the machinery we set up to deal with arbitrary orientations of a Hamilton cycle.

Let $D^*(n,p)$ denote the random digraph with vertex set $[n]$ where each possible pair of edges $\forw{uv}$ and $\back{uv}$ are included together, independently of other edges, with probability $p$. In this way $D^*(n,p)$ is the same as the binomial random graph $G(n,p)$ where we replace every undirected edge with a bidirected edge. Via a coupling argument from~\cite{McDia,mont}, to prove that $D_0 \cup D(n,C/n)$ contains $\mathcal{C}$ with probability at least $1-e^{-n}$, it suffices to show that $D_0 \cup D^*(n,C/n)$ contains $\mathcal{C}$ with probability at least $1-e^{-n}$; see Lemma~\ref{lem::coupling} for the precise statement.
This latter goal will be achievable as we only need to access the randomness in $D^*(n,C/n)$ through a simple pseudorandom property that is easily shown to hold with probability at least $1-e^{-n}$; see Definition~\ref{def::pseudo}. 

Our argument applies the absorbing method, a technique that was introduced systematically by R\"odl, Ruci\'nski, and Szemer\'edi~\cite{RRS}, but that has roots in earlier work (see, e.g.,~\cite{kriv}).

\subsection{A problem with absorbing}
To highlight a key challenge we face with absorbing, we first describe a natural approach to absorbing in the case of a consistently oriented Hamilton cycle. We note  though that absorbing was not the approach used in~\cite{bfm1} to prove Theorem~\ref{bfmthm}.

In this case,
a `global absorber' in $D_0 \cup D^*(n,C/n)$ 
is a structure $A$ on a small (but linear size) vertex set with the property that for every sufficiently small set of vertices $R$, $A \cup R$ contains the consistently oriented path on $|V(A) \cup R|$ vertices
with prescribed startpoint and endpoint in $R$. If we can obtain such a structure $A$, then we can proceed as follows:
by applying the pseudorandom property
of $ D^*(n,C/n)$ we find a bidirected path $Q$ in $ D^*(n,C/n)$ disjoint from $A$ that covers almost all of the vertices not in $A$.
Let $R$ be the set of vertices consisting of the startpoint $x$ and endpoint $y$ of $Q$, together with all those vertices not in $Q$ or $A$. Using the absorbing property of $A$ we ensure that there is a consistently oriented 
path $Q_R$ on $V(A) \cup R$ with startpoint $y$ and endpoint $x$.
Joining the startpoints and endpoints of $Q$ and $Q_R$, we obtain a consistently oriented Hamilton cycle.

In this setting of consistently oriented Hamilton cycles, one can build the global absorber $A$ from a consistently oriented path $Q_A$ with the following property. Given \emph{any} very small (but linear size) collection of vertices $R$,
we can find an ordering of the vertices
$w_1,\dots,w_t$ in $R$, and disjoint edges $\forw{x_iy_i}$ along $Q_A$ for each $i \in [t]$ where (i) if $i<j$ then $\forw{x_iy_i}$ comes before 
$\forw{x_jy_j}$ on $Q_A$; (ii) $\forw{x_iw_i}$ and $\forw{w_iy_i}$ are edges in $D_0$ for all $i \in [t]$. In this case, we can
`sandwich' in $w_i$ between $x_i$ and $y_i$ on $Q_A$, for all $i \in [t]$, to obtain a consistently oriented path on $V(Q_A) \cup R$. 
One can show such an oriented path $Q_A$ exists, and this forms the heart of the global absorbing set $A$.\footnote{Further details are required to ensure the `prescribed startpoint and endpoint' property of the global absorbing set.}

For an arbitrary orientation of a Hamilton cycle $H$, one may try to modify this argument. Indeed, fix some linear size oriented path $P_H$ which is a segment of $H$. We would like to find an oriented path $Q_A$ in $D_0 \cup D^*(n,C/n)$ that has the property that after adding any very small arbitrary set $R$ of vertices to $V(Q_A)$, there is a copy of $P_H$ precisely covering the vertices in $V(Q_A) \cup R$.

To illustrate the difficulty for arbitrary orientations, choose two very small sets of vertices $R$ and $R'$, both of which contain some fixed vertex $w$. Suppose we have constructed a path $Q_A$ that does absorb both $R$ and $R'$ analogously to the consistently oriented case. Then depending on how we have ordered $R$ and $R'$, $w$ might have to play the role of a different vertex along $P_H$.
More precisely, suppose $w$ is the $j$th vertex in the ordering of $R$ and the $k$th vertex in the ordering of $R'$ where $j<k$. Then for $R'$ we will be sandwiching in more vertices before $w$ along $P_A$ than compared to $R$. This means that the vertex in $P_H$ that $w$ plays the role of will be different in the $R$ and $R'$ cases. In particular, perhaps in the $R$ case, $w$ will need to play the role of a vertex in $P_H$ with outdegree $2$, whilst in the $R'$ case, $w$ will need to play the role of a vertex with indegree $2$.
Furthermore, this cascading effect also means a vertex along $Q_A$ may have to play the role of a different vertex in $P_H$ depending on the choice of $R$.

Of course, this would not be an issue if all the edges considered were bidirected. In that case, no matter where we sandwich in the vertices of $R$ or $R'$ in $Q_A$, we have all the necessary edges to find a copy of $P_H$, no matter how $P_H$ is oriented. Note that $D^*(n,C/n)$ by itself is too sparse to guarantee such a structure. For example, a.a.s. $D^*(n,C/n)$ does not contain a triangle containing a fixed vertex $w$, and if we were to sandwich $w$ between two consecutive vertices $x_i$ and $y_i$ along $Q_A$, then $x_i,y_i,w$ must form a triangle.
Moreover, $D_0$ may not contain any bidirected edges at all.
However, it turns out that we can guarantee that \emph{almost} all the edges along $Q_A$ are from $D^*(n,C/n)$ and so are bidirected. 
The problem is that we will have to take the edges between $R$ and $Q_A$ to be deterministic, that is, from $D_0$. 

% If there are many suitable edges between a vertex $w$ and $Q_A$ in $D_0$, then this will ensure that there are many pairs of consecutive vertices $x_i,y_i$ along $Q_A$ which we can sandwich $w$ between. This then gives us some choice about how many other vertices we absorb before $w$ along the path $Q_A$ and therefore could potentially give us the freedom to restrict which vertices on $P_H$ we require $w$ to play the role of.
If there are many pairs of consecutive vertices $x_i,y_i$ along $Q_A$ which we can sandwich $w$ between, then this gives us some choice about how many other vertices we absorb before $w$ along the path $Q_A$, potentially giving us the freedom to restrict which vertices of $P_H$ we require $w$ to play the role of.
However, in our situation, $D_0$ may not be very dense, so in general it is not the case that there is a choice of $Q_A$ so that for every vertex $w$ outside of $Q_A$, there are enough edges between $w$ and $Q_A$ in $D_0$ for this strategy to work.

As explained shortly, we will get around this problem by constructing $Q_A$ in a more sophisticated way so that ($\alpha$) $Q_A$ is only used to absorb \emph{certain} vertices, and ($\beta$) $Q_A$ has some in-built structure so that
if we absorb a vertex $w$, it must \emph{always} play the role of one of only a constant number of vertices along the path $P_H$ in $H$, no matter what the set of vertices $R$ actually is. In particular, ($\beta$) ensures that we do not need bidirected edges between $R$ and $Q_A$; rather, for a constant number of pairs of consecutive vertices $x_i,y_i$ along $Q_A$, we need single edges of the correct orientation between $\{x_i,y_i\}$ and $w$ so we can sandwich $w$ in between the two.

\subsection{Montgomery's absorbing method}
Montgomery~\cite{monty, mont_embedding} introduced an approach to absorbing that has already found a number of applications, for example, to spanning trees in random graphs~\cite{monty}, decompositions of Steiner triple systems~\cite{ferberkwan}, and tilings in randomly perturbed graphs~\cite{hmt}.
The basic idea of the method is to build a global absorber using a special graph $H_m$ as a framework. The  bipartite graph $H_m$ has a  bounded maximum degree with vertex classes $X \cup Y$ and $Z$, and has the property that if one deletes \emph{any} set of vertices of a given size from $X$, then the resulting graph contains a perfect matching; see Lemma~\ref{lem::Mont-graph}.

Roughly speaking, a global absorber is usually built from $H_m$ as follows: every edge $xy$ in $H_m$ is `replaced' with a `local absorber' $A_{xy}$ in such a way that all such absorbers $A_{xy}$ are vertex-disjoint. Here a local absorber $A_{xy}$ is some small gadget that can absorb
a certain (constant size) set of vertices $S_{xy}$.

A reason why this approach has found many applications is that, in some sense, it allows one to construct a global absorber in the case when one can only find `few' local absorbers, where what is meant by `few' here depends on the precise setting.

In the proofs of Theorems~\ref{thm::semideg_simple} and~\ref{thm:main} we will use $H_m$ again as a framework to build a global absorber. The
reason we use $H_m$, however, is different from most applications of the method (although morally the reason is similar to why Montgomery used this method in~\cite{monty}). In particular, the key idea is that one can use this framework as a way of guaranteeing property ($\beta$) above. More precisely, in our case we will replace \emph{every} edge in $H_m$ incident to $z \in Z$ with the \emph{same} local absorbing gadget $A_z$. Here $A_z$ is not designed to absorb a fixed set of vertices like before; rather, it has some local flexibility about what vertices it will absorb; see Definition~\ref{def::local_abs}. The idea is that constructing the global absorber in this way gives us the flexibility to know in advance precisely which (constant size) set of vertices on $P_H$ an absorbed vertex $w$ can play the role of. 

We emphasize that this version of Montgomery's method should be useful when trying to apply absorption to embed any spanning structure in a digraph that does not have some `nice' orientation.

%%%%%%%%%%%%%%%
\section{Random digraph ingredients}\label{sec::ingredients}

Recall that $D(n,p)$ is the digraph with vertex set $[n]$ where each of the $n(n-1)$ possible directed edges is present with probability $p$, independently of all other edges; $D^*(n,p)$ is the digraph with vertex set $[n]$ where each possible pair of edges $\forw{uv}$ and $\back{uv}$ are included together, independently of other edges, with probability $p$. 

We will use the following result, observed by Montgomery~\cite[Theorem 3.1]{mont} as a consequence of McDiarmid's coupling argument~\cite{McDia}. Recall that an {oriented graph} is a digraph in which there is at most one edge between any pair of vertices.

\begin{lemma}[\cite{McDia,mont}]\label{lem::coupling}
Let $p \in [0, 1]$ and $n \in \mathbb{N}$. Let $\cH$ be a set of oriented graphs with vertex set $[n]$ and let $D_0$ be a digraph with vertex set $[n]$. Then 
$$\Prob(\exists H \in \cH: H \subset D_0 \cup D(n,p)) \ge \Prob(\exists H \in \cH: H \subset D_0 \cup D^*(n,p))\,.$$
\end{lemma}

Note the direction of the inequality in the conclusion of Lemma~\ref{lem::coupling}. The obvious coupling between these two models gives
\[ \Prob(\exists H \in \cH : H \subseteq D_0 \cup D(n,p)) \leq \Prob(\exists H \in \cH : H \subseteq D_0 \cup D^*(n,2p-p^2)) ,\]
where the inequality is in the opposite direction but the edge-probabilities for the two models are different.

For our purposes, $\cH$ will consist of all possible copies of a single specific orientation of a cycle $\mathcal{C}$. Lemma~\ref{lem::coupling} says that it is at least as difficult to find $\mathcal{C}$ in $D_0 \cup D^*(n,p)$ as it is in $D_0 \cup D(n,p)$. By showing that $D_0 \cup D^*(n,p)$ contains $\mathcal{C}$ with probability at least $1-e^{-n}$, we can use a union bound to show that a.a.s. $D_0 \cup D(n,p)$ contains all our desired orientations of a cycle of every length. %See Section~\ref{sec::proof} for the details of our use of Lemma~\ref{lem::coupling}.

As is often the case with random (di)graph arguments, we only access the randomness through a particular sparse pseudorandom property.

\begin{definition}[Pseudorandom]\label{def::pseudo} \rm
For $1\le t \le n/2$, an $n$-vertex digraph $D$ is \emph{$t$-pseudorandom} if for every $U,W \subseteq V(D)$ with $|U|=|W|=t$ and $U \cap W = \emptyset$, there is an edge~$\forw{uw}$ directed from $U$ to $W$.
Moreover, if $D$ contains both~$\forw{uw}$ and $\back{uw}$ for every such $U$ and $W$, then we call it~\emph{$t$-bipseudorandom}. 
\end{definition}
Ben-Eliezer, Krivelevich, and Sudakov~\cite[Claim 4.3 and Lemma 4.4]{BKS} proved versions of the following two lemmas for $t$-pseudorandom digraphs.
The proofs for the~$t$-bipseudorandom versions are identical, so we omit them.
\begin{lemma}[Connecting Lemma]\label{lem::connecting}
Suppose that $D$ is a $t$-bipseudorandom digraph and  $B_1, \dots, B_\ell \subseteq V(D)$ are  pairwise disjoint sets with $|B_i| \geq 2t$ for every $i \in [\ell]$. Then there is a bidirected path $(v_1, \dots, v_\ell)$ in $D$ with $v_i \in B_i$ for every $i \in [\ell]$.
\end{lemma}

\begin{lemma}\label{lem::longpath}
If $D$ is an $n$-vertex $t$-bipseudorandom digraph, then $D$ has a bidirected path on at least $n-2t$ vertices.
\end{lemma}

In order to use the previous lemmas, we observe that $D^*(n,C/n)$ is $\ep n$-bipseudorandom with very high probability. We typically assume that $\ep n$ is an integer and ignore inconsequential rounding.

\begin{lemma}\label{lem::pseudo}
Let $0<\ep <1/2$ and let $C \geq \frac{4}{\ep} \log\frac{e}{\ep}$. Then, with probability at least $1-\exp(-C\ep^2 n/2)$, the random digraph $D^*(n,C/n)$ is $\ep n$-bipseudorandom.
\end{lemma}

\begin{proof}
%It suffices to show that $G(n,C/n)$ is $\ep n$-pseudorandom. 
Let $B_1$ and $B_2$ be disjoint subsets of vertices of size $\ep n$. In $D^*(n,C/n)$, the probability that there is no edge between $B_1$ and $B_2$ is $(1-C/n)^{(\ep n)^2}$. Taking a union bound over all possible sets $B_1$ and $B_2$ of size exactly $\ep n$, we get that the probability that there is some disjoint $B_1$ and $B_2$ with no edge between $B_1$ and $B_2$ is at most
\[ \binom{n}{\ep n}^2 \left(1-\frac{C}{n}\right)^{(\ep n)^2} \leq \exp\left( 2 \ep n \log\frac{e}{\ep} - C \ep^2 n \right) \leq \exp\left( - \frac{C \ep^2 n}{2} \right) .\qedhere\]
\end{proof}

\section{The semi-degree absorbing lemma}\label{sec::absorber}

Following the framework sketched in Section~\ref{subsec::ideas}, in this section we define and construct our global and local absorbers, Definitions~\ref{def::global_abs} and~\ref{def::local_abs}, respectively. 
Moreover, we prove the existence of many local absorbers which will then be used to construct a global absorber. 
For the latter, we use Montgomery's technique~\cite{monty, mont_embedding} based on the existence of a sparse auxiliary bipartite graph $H_m$ with `robust' matching properties; see Lemma~\ref{lem::Mont-graph}. 

In this section, we do not work in the random model; instead, our results are stated for $\eps n$-bipseudorandom digraphs with minimum semi-degree at least~$\alpha n$. In Section~\ref{sec::proof}, we apply the main absorbing lemma, Lemma~\ref{lem::global_abs}, to the randomly perturbed model to prove Theorem~\ref{thm::semideg_simple}.

\begin{definition}[Global absorber]\label{def::global_abs} \rm
Let~$P$ be an oriented path and let $D$ be a digraph.
A subset~$A \subseteq V(D)$ is called a \emph{$P$-global absorber} if for every $R \subseteq V(D) \setminus A$ such that~$\vert R\vert + \vert A\vert = \vert V(P)\vert$, and for every pair of distinct vertices~$v,v'\in R$, there is a copy of $P$ in $D[A\cup R]$ with startpoint~$v$ and endpoint $v'$.
\end{definition}

\begin{definition}[Local absorber]\label{def::local_abs} \rm
Let $P$ be an oriented path, $D$ be a digraph, $S\subseteq V(D)$, and $z\in V(D)\setminus S$.
A pair $(A,v)$ is a \emph{$P$-absorber for $(S,z)$} if
\begin{itemize}
    \item $A \subseteq V(D) \setminus (S \cup \{z\})$ is a set of $|V(P)|-2$ vertices,
    \item $v \in A$,
    \item for every $s \in S$, $D[A \cup \{s,z\}]$ contains a copy of $P$ with startpoint $v$ and endpoint $z$.
\end{itemize}
We call $v$ the \emph{startpoint} of the $P$-absorber $(A,v)$.
\end{definition}

The next lemma guarantees the existence of local absorbers with prescribed startpoint avoiding any small set of vertices --- this ensures that all the local absorbers we find will be vertex-disjoint. Before this, we prove a simple consequence of the pseudorandom property that will be useful later. Recall that for an oriented path $P = (u_1, \dots, u_k)$, $\sigma(u_i u_{i+1}) = +$ if $\forw{u_i u_{i+1}} \in E(P)$ and $\sigma(u_i u_{i+1}) =-$ otherwise.
\begin{prop}\label{prop::connecting}
Let $n, t \in \mathbb N$ where $1\leq t <n/2$.
Suppose that $D$ is an $n$-vertex~$t$-bipseudorandom digraph with~$\delta^0(D)\geq 2t + 1$.
For every oriented path $P$ on $3$ edges and every distinct $v, v' \in V(D)$, there is a copy of $P$ in $D$ with startpoint $v$ and endpoint $v'$.
\end{prop}

\begin{proof}
Let $P = (u_1, u_2, u_3, u_4)$. Let $B_1 := N^*(v)$ for $* = \sigma(u_1 u_2)$, and let $B_2 := N^*(v')$ for $* = \sigma(u_4 u_3)$. Since $|B_1|, |B_2| \geq 2t + 1$, there exists disjoint subsets $B_1' \subseteq B_1$ and $B_2' \subseteq B_2$ with $|B_1'|, |B_2'| \ge t$ which do not contain $v$ or $v'$. Since $D$ is $t$-bipseudorandom, there exists a bidirected edge $\overleftrightarrow{v_1 v_2}$ in $D$ with $v_1 \in B_1'$ and $v_2 \in B_2'$. Then $(v, v_1, v_2, v')$ contains a copy of $P$ with startpoint $v$ and endpoint $v'$.
\end{proof}

\begin{lemma}\label{lem::lin_many}
Let~$n,k\in \mathbb N$ and $\alpha, \eps > 0$,
so that $\alpha n \geq 4k+4$ and $\alpha \geq 8 (2k+2)\eps$.
Let $D$ be an $n$-vertex $\eps n$-bipseudorandom digraph with~$\delta^0(D) \geq \alpha n$.
Let~$U \subseteq V(D)$ so that $|U| \leq \alpha n/2$, and let $v \in V(D) \setminus U$.
For every oriented path~$P$ on $2k+5$ vertices, every vertex set~$S\subseteq V(D)\setminus \{v\}$ of size $k$, and every vertex~$z\in V(D) \setminus (S\cup \{v\})$, there exists a~$P$-absorber $(A,v)$ for~$(S, z)$ disjoint from $U$.%, with~$2k+3$ vertices.
\end{lemma}

\begin{proof} 
Fix~$P=(u_1, \dots, u_{2k+5})$ and  an ordering $s_1, \dots, s_{k}$ of $S$. We will find a $P$-absorber for $(S,z)$ by applying  Lemma~\ref{lem::connecting} to various neighborhoods of the $s_i$, $v$ and $z$. For each $i \in [k]$, define
\begin{itemize}
    \item $B_1 := N^*(v)$, where $* = \sigma(u_1 u_2)$,
    \item $B_{2i} := N^*(s_i)$, where $* = \sigma(u_{2i+2} u_{2i+1})$,
    \item $B_{2i+1} := N^*(s_i)$, where $* = \sigma(u_{2i+2} u_{2i+3})$,
    \item $B_{2k+2} := N^*(z)$, where $* = \sigma(u_{2k+5} u_{2k+4})$.
\end{itemize}
Because of the minimum degree condition, $\vert B_i\vert \geq \alpha n$, and therefore, we may take pairwise disjoint subsets~$B_i'\subseteq B_i\setminus (U\cup S \cup \{z\})$ such that
$$\vert B_i'\vert \geq 
\frac{(\alpha n -\tfrac{\alpha}{2}n-k-1)}{2k+2} \geq
\frac{\alpha n}{4(2k+2)} \geq 2\eps n.$$
Hence, an application of Lemma~\ref{lem::connecting} yields a bidirected path~$(v_1,\dots, v_{2k+2})$ such that~$v_i\in B'_i$ for every~$i\in [2k+2]$.

Let $A := \{v, v_1, \dots, v_{2k+2}\}$. To see that $(A,v)$ is a $P$-absorber for $(S,z)$, observe that for every $s_i\in S$, the pairs~$(u_{2i+1},u_{2i+2})$ and $(u_{2i+2},u_{2i+3})$ in $P$ have the same directions of the underlying edges as the edges from the pairs $(v_{2i},s_i)$ and $(s_i,v_{2i+1})$, allowing $s_i$ to play the role of $u_{2i+2}$ in $P$; see Figure~\ref{fig::local_abs}. More precisely,
\[ (v, v_1, \dots, v_{2i}, s_i, v_{2i+1}, \dots, v_{2k+2}, z) \]
is a copy of $P$ in $D[A \cup \{s_i, z\}]$ with startpoint $v$ and endpoint $z$.
\end{proof}

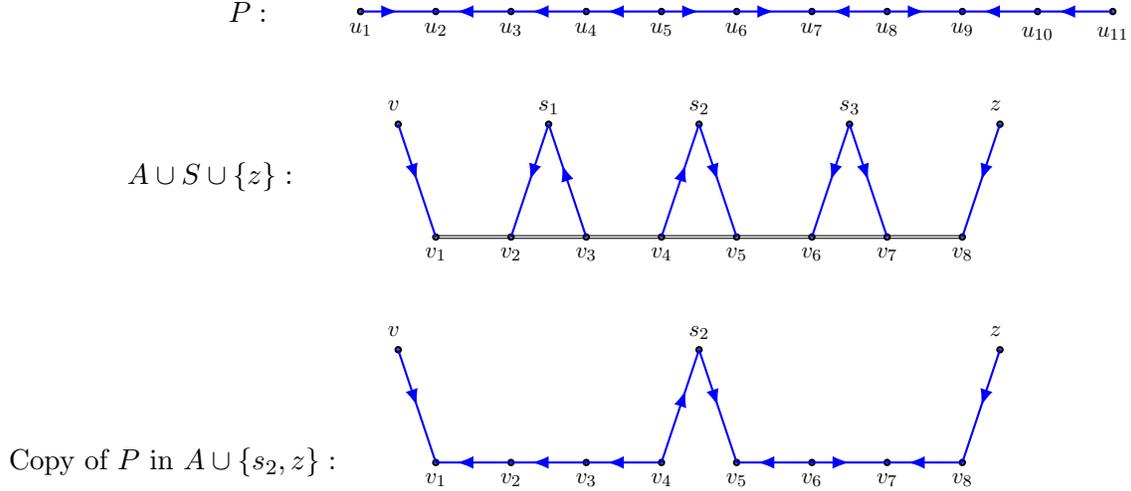
\begin{figure}[h!]
 \centering
 \begin{tikzpicture}[thick,scale=1]
 
 \node[draw=none, fill=white] at (-2.5,3) {$P:$};
 \node[draw=none, fill=white] at (-3,0.8) {$A \cup S \cup \{z\}:$};
 \node[draw=none, fill=white] at (-3.5,-3) {Copy of $P$ in $A \cup \{s_2, z\} :$};

 \foreach \y in {1,...,11} {
	 \draw (\y-2,3) node [label=below:{$\scaleto{u_{\y}}{5pt}$}] {};}
 
 \draw[middlearrow={latex}, blue] (-1,3) -- (0,3);
 \draw[middlearrow={latex reversed}, blue] (0,3) -- (1,3);
 \draw[middlearrow={latex reversed}, blue] (1,3) -- (2,3);
 \draw[middlearrow={latex reversed}, blue] (2,3) -- (3,3);
 \draw[middlearrow={latex}, blue] (3,3) -- (4,3);
 \draw[middlearrow={latex}, blue] (4,3) -- (5,3);
 \draw[middlearrow={latex reversed}, blue] (5,3) -- (6,3);
 \draw[middlearrow={latex}, blue] (6,3) -- (7,3);
 \draw[middlearrow={latex reversed}, blue] (7,3) -- (8,3);
 \draw[middlearrow={latex reversed}, blue] (8,3) -- (9,3);

 \foreach \y in {1,...,7} {
	 \draw[thin,double] (\y-1,0) -- (\y,0);
	 \draw (\y-1,0) node [label=below:{$\scaleto{v_{\y}}{5pt}$}] {};}
	\draw (7,0) node [label=below:{$\scaleto{v_{8}}{5pt}$}] {};
	
 \draw (-0.5,1.5) node [label=above:{$\scaleto{v_{\phantom{1}}}{5pt}$}] {};
 \draw (1.5,1.5) node [label=above:{$\scaleto{s_{1}}{5pt}$}] {};
 \draw (3.5,1.5) node [label=above:{$\scaleto{s_{2}}{5pt}$}] {};
 \draw (5.5,1.5) node [label=above:{$\scaleto{s_{3}}{5pt}$}] {};
 \draw (7.5,1.5) node [label=above:{$\scaleto{z_{\phantom{1}}}{5pt}$}] {};

 \draw[middlearrow={latex}, blue] (-0.5,1.5) -- (0,0);
 
 \draw[middlearrow={latex}, blue] (1.5,1.5) -- (1,0);
 \draw[middlearrow={latex reversed}, blue] (1.5,1.5) -- (2,0);
 
 \draw[middlearrow={latex reversed}, blue] (3.5,1.5) -- (3,0);
 \draw[middlearrow={latex}, blue] (3.5,1.5) -- (4,0);
 
 \draw[middlearrow={latex}, blue] (5.5,1.5) -- (5,0);
 \draw[middlearrow={latex}, blue] (5.5,1.5) -- (6,0);
 
 \draw[middlearrow={latex}, blue] (7.5,1.5) -- (7,0);

\begin{scope}[xshift=0 cm,yshift=-3 cm] 
 
 \foreach \y in {1,...,7} {
	 \draw (\y-1,0) node [label=below:{$\scaleto{v_{\y}}{5pt}$}] {};}
	\draw (7,0) node [label=below:{$\scaleto{v_{8}}{5pt}$}] {};
	
 \draw (3.5,1.5) node [label=above:{$\scaleto{s_{2}}{5pt}$}] {};

 \draw (7.5,1.5) node [label=above:{$\scaleto{z_{\phantom{1}}}{5pt}$}] {};
 
 \draw (-0.5,1.5) node [label=above:{$\scaleto{v_{\phantom{1}}}{5pt}$}] {};
 \draw[middlearrow={latex}, blue] (-0.5,1.5) -- (0,0);

  \draw[middlearrow={latex reversed}, blue] (3.5,1.5) -- (3,0);
 \draw[middlearrow={latex}, blue] (3.5,1.5) -- (4,0);
 
 \draw[middlearrow={latex}, blue] (7.5,1.5) -- (7,0);
 
 \draw[middlearrow={latex reversed}, blue] (0,0) -- (1,0);
 \draw[middlearrow={latex reversed}, blue] (1,0) -- (2,0);
 \draw[middlearrow={latex reversed}, blue] (2,0) -- (3,0);
 %\draw[middlearrow={latex}, blue] (3,0) -- (4,0);
 %\draw[middlearrow={latex}, blue] (4,0) -- (5,0);
 \draw[middlearrow={latex reversed}, blue] (4,0) -- (5,0);
 \draw[middlearrow={latex}, blue] (5,0) -- (6,0);
 \draw[middlearrow={latex reversed}, blue] (6,0) -- (7,0);
 %\draw[middlearrow={latex reversed}, blue] (8,0) -- (9,0);
\end{scope}
 \end{tikzpicture}
 \vspace{-1cm}
 \caption{
 An example of a $P$-absorber $(A,v)$ for $(S,z)$ from Lemma~\ref{lem::lin_many} with $k=|S|=3$.
 The double edges indicate that both orientations are present.
 For every $i\in\{1,2,3\}$, $A \cup \{s_i, z\}$ contains a copy of~$P$ in which vertex~$v$ plays the role of~$u_1$, vertex~$z$ plays the role of~$u_{11}$, vertex~$s_i$ plays the role of~$u_{2i+2}$, and $v_j$ plays the role of~$u_{j+1}$ for $j \leq 2i$ and $u_{j+2}$ for $j> 2i$.}
 \label{fig::local_abs}
\end{figure}

Our global absorber works in the following two-step approach: given a set $R$ of vertices we wish to absorb, we first absorb $R$ using some vertices from a specific vertex set $X$ within our global absorber; then the rest of the global absorber essentially absorbs what is left of $X$ in order to create a copy of the desired oriented path $P$. The following lemma will be used to undertake the first step of this approach; it follows easily from Lemma~\ref{lem::longpath} and Proposition~\ref{prop::connecting}.

\begin{lemma}\label{lem::connecting2}
Let~$5/n < \eps < \beta < 1$, and let $m$ and $\beta m$ be integers such that $\beta m \geq 5\ep n$.
Let $D$ be an~$n$-vertex~$\eps n$-bipseudorandom digraph, and suppose there is a set~$X\subseteq V(D)$ of size~$(1+\beta)m$ such that for every~$v\in V(D)$ and for every $* \in \{+,-\}$, $\vert N^*(v)\cap X\vert\geq 2\beta m$. Let $R\subseteq V(D)\setminus X$ be such that $|R|\geq 2$, and let $v,v' \in R$ be distinct.
Then, for every oriented path $P$ on~$|R| +\beta m$ vertices, there exists a copy of $P$ in $D[R\cup X]$ with startpoint $v$ and endpoint $v'$ that covers $R$.
\end{lemma}

\begin{proof}
Let~$X$, $R$, $v$, and $v'$ be as in the statement of the lemma. Fix an arbitrary orientation of a path $P = (u_1, \dots, u_k)$ on $k := |R| +\beta m$ vertices.  
Let $t := 2\ep n$, and $X'\subseteq X$ be an arbitrary set of size~$\beta m - 2t - 4 \geq \ep n - 4 > 0$. Let~$R_0:=R\cup X'\setminus \{v,v'\}$ and~$X_0:=X\setminus X'$. 

Observe that~$D[R_0]$ is $\eps n$-bipseudorandom, so an application of Lemma~\ref{lem::longpath} yields a bidirected path $Q_{\bi}$ on exactly
\begin{equation}\label{eq::sizeQlr}
|R_0| - 2 \ep n = |R| -2 + |X'| - t = |R| +\beta m - 3t -6
\end{equation}
vertices. Denote by $w$ and $w'$ the startpoint and endpoint of $Q_\bi$, respectively.

Label $(R_0 \setminus V(Q_\bi)) \cup \{v,w\}$ as $\{v_0 = v, v_1, \dots, v_t, v_{t+1}=w\}$. For each $0 \leq i \leq t$, we find a copy $(v_i, x_i, x_i', v_{i+1})$ of the subpath $(u_{3i+1}, u_{3i+2}, u_{3i+3}, u_{3i+4})$ of $P$, and we also find a copy $(w', x, x', v')$ of the subpath $(u_{k-3}, u_{k-2}, u_{k-1}, u_k)$ of $P$, with $x_i, x_i', x, x' \in X_0$ all distinct. This is possible by applying Proposition~\ref{prop::connecting}, observing that in total we will use  $2(t+2)$ vertices of $X_0$, and for any $U \subseteq X_0$ with $|U| \leq 2(t+2)$ and any $z, z' \in V(D) \setminus X_0$, we have
\[ \delta^0(D[(X_0 \setminus U) \cup \{z,z'\}]) \geq 2\beta m - |X'| - |U| \geq \beta m \geq 2\ep n + 1 .\]

Recall that $\circ$ denotes concatenation. Thus $(v_0, x_0, x_0', v_1, x_1, x_1', v_2, \dots, v_{t+1}) \circ Q_\bi \circ (w', x, x', v')$ contains a copy of $P$ with startpoint $v$ and endpoint $v'$, since $3t+3 + |V(Q_\bi)| + 3 = |R| + \beta m = k$ by~\eqref{eq::sizeQlr}.
\end{proof}

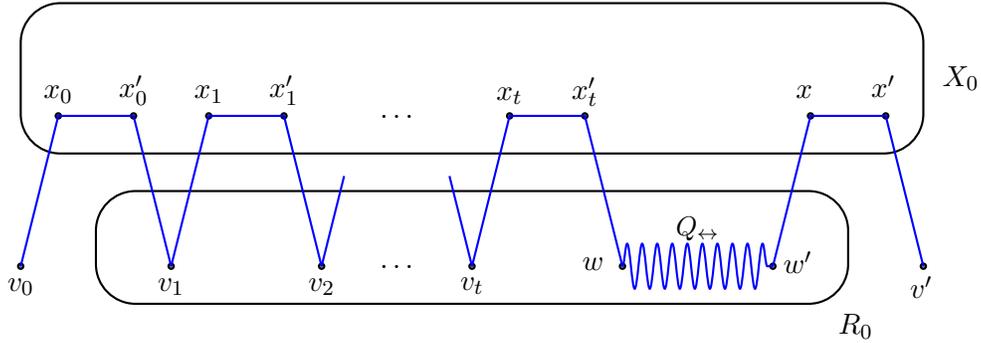
\begin{figure}[h!]
\centering
\begin{tikzpicture}[thick,scale=1]
 
\begin{scope}[xshift = -1cm, yshift = 0cm]
\node[draw=none, fill=white] at (12,2.5) {$X_0$};
\foreach \y in {0,1} {
    \draw (2*\y,2) node [label=above:{$x_{\y}$}] {};
    \draw (2*\y+1,2) node [label=above:{$x'_{\y}$}] {};}
    % \draw (2*\y,3) node [label=above:{$\scaleto{x_{\y}}{5pt}$}] {};
    % \draw (2*\y+1,3) node [label=above:{$\scaleto{x'_{\y}}{8pt}$}] {};}

\node[draw=none, fill=white] at (4.5,2) {$\ldots$};

\draw (6,2) node [label=above:{$x_{t}$}] {};
\draw (7,2) node [label=above:{$x'_{t}$}] {};

\draw (10,2) node [label=above:{$x_{\phantom{1}}$}] {};
\draw (11,2) node [label=above:{$x'_{\phantom{1}}$}] {};

\draw[rounded corners=15] (-.5,1.5) rectangle (11.5, 3.5);
\end{scope} 

\begin{scope}[xshift=-1.5cm, yshift=0cm]
\node[draw=none, fill=white] at (11.1,-0.8) {$R_0$};
\node[draw=none, fill=white] at (9, 0.5) {$\scaleto{Q_{\bi}}{9pt}$};
\draw[rounded corners=15] (1,-.5) rectangle (11,1);

\foreach \y in {1,2} {
	 \draw (2*\y,0) node [label=below:{$v_{\y}$}] {};}
\draw (0,0) node [label=below:{$v_{0}$}] {};
\node[draw=none, fill=white] at (5,0) {$\ldots$};
\draw (6,0) node [label=below:{$v_{t}$}] {};
\draw (8,0) node [label=left:{$w_{\phantom{1}}$}] {};
\draw (10,0) node [label=right:{$w'_{\phantom{1}}$}] {};
\draw (12,0) node [label=below:{$v'_{\phantom{1}}$}] {};

\draw[decorate, decoration={snake, segment length=0.2cm, amplitude=0.3cm}, blue] (8,0) -- (10,0);
\end{scope}

\foreach \x in {-2,0,4,8}{
\begin{scope}[xshift = \x cm, yshift=0cm]
    \draw[blue] (0.5,0) -- (1,2) -- (2,2) -- (2.5,0);
\end{scope}}
    \draw[blue] (2.5,0) -- (2.8,1.2);
    \draw[blue] (4.5,0) -- (4.2,1.2);

\end{tikzpicture}
\caption{
A copy of $P$ in $D[R\cup X]$ with startpoint $v$ and endpoint $v'$ that covers all the vertices in $R$, as found in Lemma~\ref{lem::connecting2}.}
\label{fig::first_absorb}
\end{figure}

The next lemma provides the sparse auxiliary bipartite graph $H_m$ with robust matching properties that we use as a framework to build our global absorber.

\begin{lemma}[\cite{mont_embedding}]\label{lem::Mont-graph}
For every~$0<\beta\leq 1$ and for sufficiently large $m \in \N$ with $\beta m \in \N$, there exists a bipartite graph $H_m$ with parts $X \dot\cup Y$ and $Z$, such that $|X|=(1+\beta)m$, $|Y| = 2m$, $|Z| = 3m$, $H_m$ has maximum degree at most $40$, and for every $X' \subseteq X$ of size $m$, there exists a perfect matching between $X' \cup Y$ and $Z$ in $H_m$.
\end{lemma}

We are now ready to prove the absorbing lemma for Theorem~\ref{thm::semideg_simple}.

\begin{lemma}[Absorbing lemma]\label{lem::global_abs}
For every $0 < \alpha, \eta \leq 1$, there exists an~$0<\eps <\alpha \eta /1000$ such that the following holds for sufficiently large $n$.
Given an oriented path~$P$ of size~$\lceil\alpha n/4\rceil$, every~$\eps n$-bipseudorandom $n$-vertex digraph $D$ with $\delta^0(D) \geq \alpha n$ contains a $P$-global absorber $A$ of size at most $\eta|V(P)|$.
\end{lemma}

\begin{proof}
Given~$\alpha$ and $\eta$, define~$p :={\eta \alpha}/{2024}$.  Set $\beta:=\alpha/10$ and $\eps:= p \beta/{6}$.
Let~$P=(u_1, \dots, u_k)$ with~$k:=\lceil \alpha n/4\rceil$.

By applying the Chernoff bound for the hypergeometric distribution, we  obtain a set $X \subseteq V(D)$  such that
\begin{enumerate}[label={($X$\arabic*)}]
 \item \label{X1} $|X| = (1+\beta)m$ with $p n  < m < |X| < 2p n$ and $\beta m \geq 5\ep n +2$, where we assume that $m$ and $\beta m$ are integers and ignore inconsequential rounding issues;
 \item \label{X2} for every vertex $v \in V(D)$ and for every $* \in \{+,-\}$,
 $|N^*(v) \cap X| \geq \frac{p}{2}|N^*(v)|\geq \frac{p\alpha n}{2} \geq 2\beta m +2$.
\end{enumerate}
Note that  $X$ satisfies the hypothesis of Lemma~\ref{lem::connecting2}, which will be used later. Arbitrarily choose two disjoint sets~$Y, Z\subseteq V(D)\setminus X$ of sizes~$2m$ and~$3m$, respectively. We form an auxiliary graph~$H_m$ isomorphic to the graph from Lemma~\ref{lem::Mont-graph} on $X \cup Y \cup Z$. We label $Z$ as $\{z_1, \dots, z_{3m}\}$, and let $N_i \subseteq X \cup Y$ be a set of size exactly $40$ containing $N_{H_m}(z_i)$.

We are now ready to construct the global absorber. We will identify a particular segment of $P$ for each $i \in [3m]$, and use Lemma~\ref{lem::lin_many} to obtain a
local absorber for~$(N_i,z_i)$ for such a segment.
These local absorbers combined will act as a global absorber since we can apply Lemma~\ref{lem::connecting2} to form the rest of $P$ with any appropriate set $R$ of vertices we wish to absorb, using exactly $\beta m$ vertices of $X$ in the process; the remaining part of $X$ along with $Y$ is matched to $Z$ via the property of $H_m$ given in Lemma~\ref{lem::Mont-graph}, and this matching will tell us how to use each local absorber; see Figure~\ref{fig::global_abs}.

\smallskip

Let $z_0 \in V(D) \setminus (X \cup Y \cup Z)$ be arbitrary. For each $i \in [3m]$, we find a $(u_{84i-80}, \dots, u_{84i+4})$-absorber $(A_i,z_{i-1})$ for $(N_i, z_i)$ such that $z_{i-1} \in A_i$, the sets $A_i$ for $i\in [3m]$ are pairwise disjoint and disjoint from $X \cup Y \cup \{z_{3m}\}$.\footnote{Note that, for each $i \in [3m]$, $z_i \not\in A_i$ and $z_{i-1} \in A_i$. Then the $A_i$'s cannot be disjoint from $Z$. Hence, we only ask them to be disjoint from $\{z_{3m}\}$.}
This is possible by applying Lemma~\ref{lem::lin_many} (with $40$ playing the role of $k$), since we require the absorbers to be disjoint from at most
\begin{equation}\label{eq::Asize}
|A| = |X \cup Y| + 1 + \sum_{i=1}^{3m} |A_i| = 1 + (3+\beta)m + 3m\cdot 83 = (252+\beta)m+1 \leq \eta |V(P)| \leq \alpha n/2
\end{equation}
vertices, where we define
\[ A := X \cup Y \cup \{z_{3m}\} \cup \bigcup_{i \in [3m]} A_i .\]
We claim that $A$ is a $P$-global absorber, which by~\eqref{eq::Asize} has size at most $\eta |V(P)|$.

Let $R \subseteq V(D) \setminus A$ be such that $|R| + |A| = |V(P)| = k$, and let $v,v' \in R$ be distinct. By~\ref{X1} and \ref{X2}, we have that $\delta^0(D[X \cup \{v, z_0\}]) \geq 2\beta m \geq 10\ep n \geq 2 \eps n +1$, so we may apply Proposition~\ref{prop::connecting} to obtain a copy $(v, x_1, x_2, z_0)$ of $(u_1, u_2, u_3, u_4)$, where $x_1, x_2 \in X$.

Let $\bar{X} := X \setminus \{x_1, x_2\}$, let $\bar{\beta} := \beta - \frac{2}{m}$, and let $\bar{R} := (R \cup \{z_{3m}\}) \setminus \{v\}$. By~\ref{X1}, $|\bar{X}| = (1+\bar{\beta})m \geq m + 5\ep n$, and by~\ref{X2}, for every $v \in V(D)$ and for every $* \in \{+,-\}$, we have $|N^*(v) \cap \bar{X}| \geq 2\beta m$. Therefore we may apply Lemma~\ref{lem::connecting2} to obtain a copy $Q$ of $(u_{252m+4}, \dots, u_k)$ in $D[\bar{X} \cup \bar{R}]$ covering $\bar{R}$ and exactly $\bar{\beta} m$ vertices of $\bar{X}$ with startpoint $z_{3m}$ and endpoint $v'$.
%\COMMENT{AT: deleted the following, as we never use this:

%checking that by~\eqref{eq::Asize}
%\[ |\bar{R}| + \bar{\beta} m = |R| + \beta m - 2 = k - |A| + \beta m - 2 = k - 252m - 3 .\]}

Now we activate the local absorbers. Let $X' := \bar{X} \setminus V(Q)$, and note that $|X'| = m$. By Lemma~\ref{lem::Mont-graph}, there exists a matching between $Z$ and $X' \cup Y$ in $H_m$. Fixing such a matching, let $z_i' \in N_i$ be the vertex matched to $z_i$ for each $i \in [3m]$. Since $(A_i, z_{i-1})$ is a $(u_{84i-80}, \dots, u_{84i+4})$-absorber for $(N_i, z_i)$, there exists a copy $Q_i$ of $(u_{84i-80}, \dots, u_{84i+4})$ in $D[A_i \cup \{z_i, z_i'\}]$ with startpoint $z_{i-1}$ and endpoint $z_i$. Concatenating as
\[ (v, x_1, x_2, z_0) \circ Q_1 \circ \cdots \circ Q_{3m} \circ Q ,\]
we obtain a copy of $P$ in $D[A \cup R]$ with startpoint $v$ and endpoint $v'$.
\end{proof}

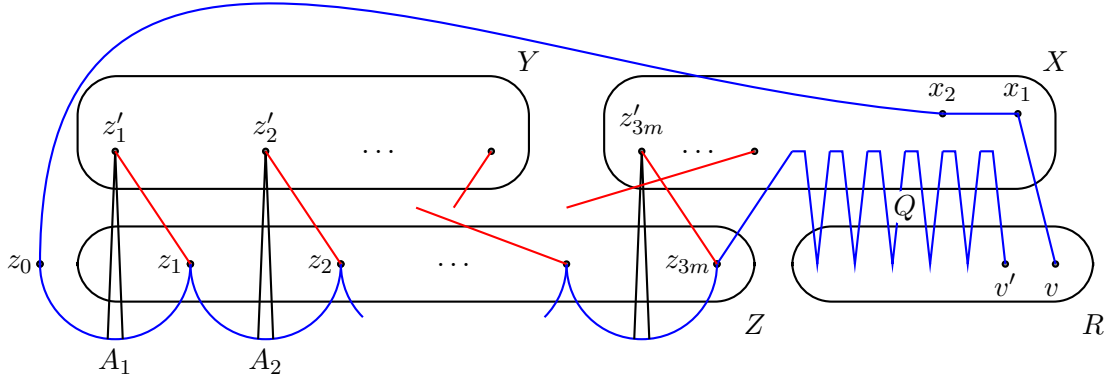
\begin{figure}[h!]
\centering
\begin{tikzpicture}[thick,scale=1]

%\useasboundingbox (-2.5,-2) rectangle (13,6.5);
 
% Y
\begin{scope}[xshift = -1cm, yshift = -1cm]
\node[draw=none, fill=white] at (5.5,3.7) {$Y$};
\draw (0,2.5) node[label=above:{$z_1'$}] {};
\draw (2,2.5) node[label=above:{$z_2'$}] {};
\draw (5,2.5) node[] {};

\node[draw=none, fill=white] at (3.5,2.5) {$\ldots$};

\draw[rounded corners=15] (-.5,2) rectangle (5.5,3.5);
\end{scope} 

% X
\begin{scope}[xshift = 6cm, yshift = -1cm]
\node[draw=none, fill=white] at (5.5,3.7) {$X$};
\draw (0,2.5) node [label=above:{$z_{3m}'$}] {};
\node[draw=none, fill=white] at (.75,2.5) {$\ldots$};
%\draw (1,3) node [] {$\cdots$};
\draw (1.5,2.5) node [] {};
\draw (4,3) node [label=above:{$x_2$}] {};
\draw (5,3) node [label=above:{$x_1$}] {};

\draw[rounded corners=15] (-.5,2) rectangle (5.5,3.5);
\end{scope} 

% Z
\begin{scope}[xshift = -2cm, yshift = 0cm]
\node[draw=none, fill=white] at (9.5,-0.8) {$Z$};
\foreach \y in {0,2,4,7,9} {
    \draw (\y,0) node [] {};}

\foreach \y in {0,1,2} {   
\draw (2*\y,0) node [label=left:{$z_{\y}$}] {};}
\draw (9,0) node [label=left:{$z_{3m}$}] {};

\node[draw=none, fill=white] at (5.5,0) {$\ldots$};

\draw[rounded corners=15] (.5,-.5) rectangle (9.5,.5);
\end{scope} 

% R
\begin{scope}[xshift = 7.5 cm, yshift = 0cm]
\node[draw=none, fill=white] at (4.5,-0.8) {$R$};
\draw (3+1/3,0) node [label=below:{$v'$}] {};
\draw (4,0) node [label=below:{$v\phantom{'}$}] {};

\draw[rounded corners=15] (.5,-.5) rectangle (4.5,.5);
\end{scope} 

% Q
\foreach \x in {0,.5,1,1.5,2}{
\draw[blue] (8+\x,1.5) -- (8+1/6+\x,1.5) -- (8+1/3+\x,0) -- (8+\x+1/2,1.5);
}
\draw[blue] (7,0) -- (8,1.5);
\draw[blue] (8+2.5,1.5) -- (8+1/6+2.5,1.5) -- (8+1/3+2.5,0);
\node[draw=none, fill=white] at (9.5,.75) {$Q$};
%\draw[decorate, decoration={snake, segment length=0.36cm, amplitude=8mm}, blue] (9,.5) -- (11,.5);
%\draw[blue] (10.9,2) -- (11,0);

%z_0 to v
\draw[blue] (-2,0) .. controls (-2,6) and (4,2.5) .. (10,2);
\draw[blue] (10,2) -- (11,2) -- (11.5,0);

%local absorbers: A_i
\foreach \x in {-2,0,5}{
\draw[thick, blue] (\x,0) arc (180:360:1);}
\draw[thick, blue] (2,0) arc (180:225:1);
\draw[thick, blue] (5,0) arc (360:315:1);

\draw[thick] (0.9,-1) -- (1,1.5) -- (1.1,-1);
\draw[thick] (-0.9,-1) -- (-1,1.5) -- (-1.1,-1);
\draw[thick] (5.9,-1) -- (6,1.5) -- (6.1,-1);

\node[draw=none, fill=white] at (-1,-1.3) {$A_1$};
\node[draw=none, fill=white] at (1,-1.3) {$A_2$};
\node[draw=none, opacity=0, text opacity=1] at (6,-1.3) {$A_{3m}$};

%matching edges
\draw[thick, red] (0,0) -- (-1,1.5);
\draw[thick, red] (2,0) -- (1,1.5);
\draw[thick, red] (3.5,.75) -- (4,1.5);
\draw[thick, red] (7,0) -- (6,1.5);
\draw[thick, red] (5,0) -- (3,.75);
\draw[thick, red] (5,.75) -- (7.5,1.5);
\end{tikzpicture}

\caption{
The global absorber. The blue path with startpoint $v$ and endpoint $v'$ is a copy of $P$ covering $R$, $X$, $Y$, and $Z$. The red edges are the matching in the auxiliary graph $H_m$, dictating which local absorber $A_i$ to use for each vertex in $X \cup Y$.}
\label{fig::global_abs}
\end{figure}

\section{Proof of Theorem~\ref{thm::semideg_simple}}\label{sec::proof}

\begin{proof}[Proof of Theorem~\ref{thm::semideg_simple}]
Given $\alpha>0$, let~$\eps>0$ be as in Lemma~\ref{lem::global_abs} on input $\alpha$ and~$\eta := 1/2$. Set~$C := 4e/\ep^2$. 
Let $D_0$ be an $n$-vertex digraph with~$\delta^0(D_0)\geq \alpha n$.

Given any orientation of a cycle $\mathcal C$ of length between $3$ and $n$, our first aim is to prove that $D:=D_0\cup D^*(n, C/n)$ contains a copy of~$\mathcal C$ with probability at least $1-e^{-n}$.
Note that Lemma~\ref{lem::pseudo} implies that $D^*(n, C/n)$ is~$\eps n$-bipseudorandom with probability at least $1-e^{-n}$; thus, we may assume that $D$ is $\eps n$-bipseudorandom.

If $|\mathcal C| = 3$, we use a proof similar to that of Proposition~\ref{prop::connecting}. Fix a vertex $v \in V(D)$, and consider $N^+_D(v)$ and $N^-_D(v)$, which both have size at least $\alpha n \geq 2\ep n$. Between  these sets there is a bidirected edge, giving both possible orientations of $\mathcal C$.

If~$4 \leq \vert \mathcal C\vert \leq \alpha n/2$, then we apply Lemma~\ref{lem::longpath} to find a bidirected path $Q_\bi$ in $D$ on~$\vert \mathcal C\vert - 2$ vertices.
Let~$v$ and $v'$ be the startpoint and endpoint of $Q_\bi$, respectively, and let $P$ be a subpath of $\mathcal C$ on $3$ edges.
Observe that~$\delta^0(D[(V(D)\setminus V(Q_\bi)) \cup \{v,v'\}])\geq \alpha n /2 \geq 2\ep n+1$, and hence we may apply Proposition~\ref{prop::connecting} to find a copy $Q$ of $P$ in $D$ with startpoint $v'$ and endpoint $v$. Joining $Q$ and $Q_\bi$ at both ends, we obtain a copy of $\mathcal C$ in $D$.

If $|\mathcal C| \geq \alpha n/2$, then let $P$ be a subpath of $\mathcal{C}$ on $\ceil{\alpha n/4}$ vertices. We apply Lemma~\ref{lem::global_abs} to find a~$P$-global absorber~$A$ of size at most~$\ceil{\alpha n/8}$.
Since $D[V(D)\setminus A]$ is $\eps n$-bipseudorandom, Lemma~\ref{lem::longpath} yields a bidirected path on at least $n - |A| - 2\eps n > n-|P|+2$  vertices disjoint from $A$. 
Ignoring some vertices, let $Q_\bi$ be a bidirected path on~$\vert \mathcal C\vert-|P|+2$ vertices in $D[V(D) \setminus A]$, and let $v$ and $v'$ be the startpoint and endpoint of $Q_\bi$, respectively.
Let $R \subseteq (V(D) \setminus (V(Q_\bi) \cup A)) \cup \{v, v'\}$ with $v,v' \in R$ and $|R| = |P| - |A|$. By Definition~\ref{def::global_abs}, there is a copy $Q$ of $P$ in $D$ with startpoint $v'$ and endpoint $v$ covering exactly the vertices of $R\cup A$. Joining $Q$ and $Q_\bi$ at both ends, we obtain a copy of $\mathcal C$ in $D$.

Thus, for every orientation of a cycle $\mathcal C$ of length between $3$ and $n$ we have that $D_0 \cup D^*(n,C/n)$ contains a copy of $\mathcal C$ with probability at least $1-e^{-n}$. By Lemma~\ref{lem::coupling}, $D_0 \cup D(n,C/n)$ contains a copy of $\mathcal C$ with probability at least $1-e^{-n}$. Taking a union bound over all $n-2$ possible lengths and all at most $2^n$ possible orientations of each length, we have that $D_0 \cup D(n,C/n)$ contains every orientation of a cycle of length between $3$ and $n$ a.a.s.

Finally, to see that $D_0 \cup D(n,C/n)$ contains a cycle of length $2$ a.a.s., simply observe that $D_0$ has at least $\alpha n^2$ directed edges, and so the probability no edge of $D(n,C/n)$ is in the opposite orientation of an edge of $D_0$ is at most
\[ (1-C/n)^{\alpha n^2} \leq e^{- C\alpha n} .\qedhere\]
\end{proof}

\section{The total degree absorbing lemma}\label{sec::total}

While following the same general outline as the proof of Theorem~\ref{thm::semideg_simple}, the proof of Theorem~\ref{thm:main} requires several more details in order to deal with complications arising from two sources. First, since the statement of  Theorem~\ref{thm:main} would be false if
we relaxed it to a statement about arbitrary orientations of cycles, our proof needs to exploit 
the property that the cycles we wish to embed do not have $(1-o(1))n$ vertices of indegree $1$.
Second, the condition $\delta(D) \geq 2 \alpha n$ is only enough to give that $d^+(v) \geq \alpha n$ or $d^-(v) \geq \alpha n$, but not necessarily both, for each vertex $v \in V(D)$. After introducing some convenient notation, we redefine the global and local absorbers from Section~\ref{sec::absorber} to fit our needs here. The statements of the absorbing lemma and helper lemmas are very similar to those in Section~\ref{sec::absorber}, of course with additional technicalities. 
%We suggest that the interested reader familiarize themselves with Section~\ref{sec::absorber} before attempting this one.

Let $P = (u_1, \dots, u_k)$ be an oriented path. Recall that we call $u_1$ the \emph{startpoint} of $P$ and $u_k$ the \emph{endpoint} of $P$, and recall that $\sigma(u_i u_{i+1}) = +$ if $\forw{u_i u_{i+1}} \in E(P)$ and $\sigma(u_i u_{i+1}) = -$ otherwise. For $i \in [k-1]\setminus \{1\}$, we call $u_i$ a \emph{swap vertex} of $P$ if the indegree of $u_i$ in $P$ is $0$ or $2$. At swap vertices, the directions of the edges of an oriented path change from forwards to backwards, or vice versa. Note that the `type' of swap vertices alternate along the path between indegree $0$ and $2$. Recall that when the endpoint of $P$ is the startpoint of $P'$ and the oriented paths $P$ and $P'$ are otherwise vertex-disjoint,  $P \circ P'$ denotes the concatenation of the two paths.

We cannot hope to find a copy of a given oriented path with prescribed startpoint and endpoint in a digraph $D$ unless those vertices have suitably high in- or outdegree in $D$. This motivates the following two definitions.

\begin{definition}[$\alpha$-compatible]\label{def::total:compatible} \rm
Let $P = (u_1, \dots, u_k)$ be an oriented path, let $D$ be an $n$-vertex digraph, and let $\alpha > 0$. For $v_1,v_k \in V(D)$, we say that $(v_1,v_k)$ is \emph{$\alpha$-compatible with $P$} if $d^*(v_1) \geq \alpha n$ for $* = \sigma(u_1 u_2)$ and $d^*(v_k) \geq \alpha n$ for $* = \sigma(u_k u_{k-1})$.
\end{definition}

\begin{definition}[Global absorber]\label{def::total:global} \rm
Let $P$ be an oriented path, let $D$ be a digraph, and let $\alpha > 0$. A subset $A \subseteq V(D)$ is  a \emph{$(P,\alpha)$-global absorber} if for every $R \subseteq V(D) \setminus A$ such that $|R| + |A| = |V(P)|$, and for every pair of distinct $v,v' \in R$ such that $(v,v')$ is $\alpha$-compatible with $P$, there is a copy of $P$ in $D[A \cup R]$ with startpoint $v$ and endpoint $v'$.
\end{definition}

As in Section~\ref{sec::absorber}, the global absorber will be constructed out of smaller units called local absorbers, defined in Definition~\ref{def::total:local}. We use a slightly expanded definition of local absorber as compared to Definition~\ref{def::local_abs}, so that we have the added flexibility of specifying the endpoint of the local absorber.

%The reader familiar with Section~\ref{sec::absorber} may see two ways of a building a global absorber from local absorbers. We could break up $P$ at specially selected swap vertices, which will be the locations of the vertices of $Z$, between which we place local absorbers; one potential obstacle is dealing with the potentially large distances between swap vertices in $P$. Instead, we break up $P$ into constant size pieces of all the same size, some of which will become local absorbers. This requires us to terminate the copies of these pieces at vertices which are compatible with the next piece, motivating the following definition.

\begin{definition}[Local absorber]\label{def::total:local} \rm
Let $P$ be an oriented path, $D$ be a digraph, $S \subseteq V(D)$, and $z \in V(D) \setminus S$. A triple $(A,v,v')$ is a \emph{$P$-absorber for $(S,z)$} if
\begin{itemize}
    \item $A \subseteq V(D) \setminus (S \cup \{z\})$ is a set of $|V(P)|-2$ vertices,
    \item $v,v' \in A$, with $v \neq v'$,
    \item for every $s \in S$, $D[A \cup \{s,z\}]$ contains a copy of $P$ with startpoint $v$ and endpoint $v'$.
\end{itemize}
We call $v$ the \emph{startpoint} and $v'$ the \emph{endpoint} of the $P$-absorber $(A,v,v')$.
\end{definition}

The next lemma guarantees the existence of local absorbers avoiding some small set of vertices --- this ensures that all the local absorbers we find will be vertex-disjoint.

\begin{lemma}\label{lem::total:local}
Let $n,k, \ell \in \mathbb N$ and $\eps,\alpha >0$ so that 
$\frac{1}{n} \leq \ep \leq \frac{\alpha}{8k}$ and $k \geq 3\ell +9$. Let $D$ be an $n$-vertex $\ep n$-bipseudorandom digraph with $\delta(D) \geq 2\alpha n$. Let $U \subseteq V(D)$ so that  $|U| \leq \alpha n/2$, and let $v,v' \in V(D) \setminus U$ be distinct. Let $P$ be an arbitrary   oriented path $P$ on $k$ vertices with at least $3\ell+7$ swap vertices such that $(v,v')$ is $\alpha$-compatible with $P$. For every $S \subseteq V(D) \setminus \{v,v'\}$ of size $\ell$, and every vertex $z \in V(D) \setminus (S \cup \{v,v'\})$, there exists a $P$-absorber $(A,v,v')$ for $(S,z)$ disjoint from $U$.
\end{lemma}

\begin{proof}
Let $P = (u_1, \dots, u_k)$. Label $S \cup \{z\}$ as $z_1, \dots, z_{\ell+1}$, where $z_{\ell+1} := z$. We will find a $P$-absorber for $(S,z)$ by applying  Lemma~\ref{lem::connecting} to various neighborhoods of the $z_i$, $v$ and $v'$. Let $*_i := +$ if $d^+(z_i) \geq \alpha n$, and let $*_i := -$ otherwise. Choose $\ell+1$ swap vertices of $P$, $u_{i_1}, \dots, u_{i_{\ell+1}}$, such that
\begin{itemize}
    \item $i_{j+1} - i_j \geq 2$ for $0 \leq j \leq \ell+1$, where $i_0 := 2$ and $i_{\ell+2} := k-1$,
    \item $i_{\ell+1} - i_{\ell} \geq 3$,
    \item $d_P^{*_j}(u_{i_j}) = 2$ for every $j \in [\ell+1]$.
\end{itemize}
This is possible because $P$ has at least $3\ell+7$ swap vertices, and they alternate having in- or outdegree $2$. Define
\begin{itemize}
    \item $B_1: = N^*(v)$, where $* = \sigma(u_1 u_2)$,
    \item $B_{i_j-2} := B_{i_j-1} := N^{*_j}(z_j)$ for $j \in [\ell]$,
    \item $B_{i_{\ell+1}-3} := B_{i_{\ell+1}-2} := N^{*_{\ell+1}}(z_{\ell+1})$,
    \item $B_{k-4} := N^*(v')$, where $* = \sigma(u_k u_{k-1})$,
    \item $B_i := V(D)$ for all remaining $i\in [k-4]$.
\end{itemize}
Since $(v,v')$ is $\alpha$-compatible with $P$, and by the definition of the $*_j$, we have that $|B_i| \geq \alpha n$ for every $i \in [k-4]$. Since $|U| \leq \alpha n/2$, there exists pairwise disjoint subsets $B_i' \subseteq B_i\setminus (U \cup S \cup \{z\})$ such that for all $i \in [k-4]$,
$$|B_i'| \geq \frac{(\alpha n- \alpha n/2 -\ell -1)}{(k-4)} \geq 
\frac{\alpha n}{4(k-4)} \geq
2 \ep n.$$
Lemma~\ref{lem::connecting} gives a bidirected path $(v_1, \dots, v_{k-4})$ in $D$ with $v_i \in B' _i$ for every $i \in [k-4]$. Let $A := \{v, v_1, \dots, v_{k-4}, v'\}$, and note that $A$ is disjoint from $U$ and $S \cup \{z\}$.

To see that $A$ is a $P$-absorber for $(S,z)$, note that for every $z_j \in S$, the path
\[ (v, v_1, \dots, v_{i_j-2}, z_j, v_{i_j-1}, \dots, v_{i_{\ell+1}-3}, z, v_{i_{\ell+1}-2}, \dots, v_{k-4}, v') \]
is a copy of $P$ in $D[A \cup \{z_j,z\}]$ with startpoint $v$ and endpoint $v'$; see Figure~\ref{fig::local_abs_total_deg}.
\end{proof}

\begin{figure}[h!]
 \centering
 \begin{tikzpicture}[thick,scale=1]
 
 \node[draw=none, fill=white] at (-3.5,3) {$P:$};
 \node[draw=none, fill=white] at (-3.5,0.8) {$A \cup S \cup \{z\}:$};
 \node[draw=none, fill=white] at (-3.5,-3) {Copy of $P$ in $A \cup \{s_2, z\} :$};

\begin{scope}[xshift=-1cm, yshift=0cm]
    \foreach \y in {1,...,13} {
	 \draw (\y-2,3) node [label=below:{$\scaleto{u_{\y}}{5pt}$}] {};}
    
    \draw[middlearrow={latex}, blue] (-1,3) -- (0,3);
    \draw[middlearrow={latex reversed}, blue] (0,3) -- (1,3);
    \draw[middlearrow={latex reversed}, blue] (1,3) -- (2,3);
    \draw[middlearrow={latex}, blue] (2,3) -- (3,3);
    \draw[middlearrow={latex reversed}, blue] (3,3) -- (4,3);
    \draw[middlearrow={latex}, blue] (4,3) -- (5,3);
    \draw[middlearrow={latex reversed}, blue] (5,3) -- (6,3);
    \draw[middlearrow={latex reversed}, blue] (6,3) -- (7,3);
    \draw[middlearrow={latex reversed}, blue] (7,3) -- (8,3);
    \draw[middlearrow={latex}, blue] (8,3) -- (9,3);
    \draw[middlearrow={latex}, blue] (9,3) -- (10,3);
    \draw[middlearrow={latex reversed}, blue] (10,3) -- (11,3);  
\end{scope}

 \foreach \y in {1,...,8} {
	 \draw[thin,double] (\y-1,0) -- (\y,0);
	 \draw (\y-1,0) node [label=below:{$\scaleto{v_{\y}}{5pt}$}] {};}
	\draw (8,0) node [label=below:{$\scaleto{v_{9}}{5pt}$}] {};
	
 \draw (-0.5,1.5) node [label=above:{$\scaleto{v_{\phantom{1}}}{5pt}$}] {};
 \draw (1.5,1.5) node [label=above:{$\scaleto{s_{1}}{5pt}$}] {};
 \draw (4.5,1.5) node [label=above:{$\scaleto{s_{2}}{5pt}$}] {};
 \draw (6.5,1.5) node [label=above:{$\scaleto{z_{\phantom{1}}}{5pt}$}] {};
 \draw (8.5,1.5) node [label=above:{$\scaleto{v_{\phantom{1}}'}{8pt}$}] {};

 \draw[middlearrow={latex}, blue] (-0.5,1.5) -- (0,0);
 
 \draw[middlearrow={latex}, blue] (1.5,1.5) -- (1,0);
 \draw[middlearrow={latex}, blue] (1.5,1.5) -- (2,0);
 
 \draw[middlearrow={latex reversed}, blue] (4.5,1.5) -- (4,0);
 \draw[middlearrow={latex reversed}, blue] (4.5,1.5) -- (5,0);
 
 \draw[middlearrow={latex}, blue] (6.5,1.5) -- (6,0);
 \draw[middlearrow={latex}, blue] (6.5,1.5) -- (7,0);
 
 \draw[middlearrow={latex}, blue] (8.5,1.5) -- (8,0);

\begin{scope}[xshift=0 cm,yshift=-3 cm] 
 
    \foreach \y in {1,...,8} {
     \draw (\y-1,0) node [label=below:{$\scaleto{v_{\y}}{5pt}$}] {};}
    \draw (8,0) node [label=below:{$\scaleto{v_{9}}{5pt}$}] {};
    
    \draw (4.5,1.5) node [label=above:{$\scaleto{s_{2}}{5pt}$}] {};
    \draw (6.5,1.5) node [label=above:{$\scaleto{z_{\phantom{1}}}{5pt}$}] {};
    \draw (8.5,1.5) node [label=above:{$\scaleto{v_{\phantom{1}}'}{8pt}$}] {};
    
    \draw (-0.5,1.5) node [label=above:{$\scaleto{v_{\phantom{1}}}{5pt}$}] {};
    
    \draw[middlearrow={latex}, blue] (-0.5,1.5) -- (0,0);
    
    \draw[middlearrow={latex reversed}, blue] (4.5,1.5) -- (4,0);
    \draw[middlearrow={latex reversed}, blue] (4.5,1.5) -- (5,0);
    
    \draw[middlearrow={latex}, blue] (6.5,1.5) -- (6,0);
    \draw[middlearrow={latex}, blue] (6.5,1.5) -- (7,0);
    
    \draw[middlearrow={latex}, blue] (8.5,1.5) -- (8,0);
 
    \draw[middlearrow={latex reversed}, blue] (0,0) -- (1,0);
    \draw[middlearrow={latex reversed}, blue] (1,0) -- (2,0);
    \draw[middlearrow={latex}, blue] (2,0) -- (3,0);
    \draw[middlearrow={latex reversed}, blue] (3,0) -- (4,0);
    \draw[middlearrow={latex reversed}, blue] (5,0) -- (6,0);
    \draw[middlearrow={latex}, blue] (7,0) -- (8,0);
\end{scope}
 \end{tikzpicture}
 \vspace{-1cm}
 \caption{
 An example of a $P$-absorber $(A,v,v')$ for $(S=\{s_1,s_2\},z)$ from Lemma~\ref{lem::total:local}.
 The double edges indicate that both orientations are present.
 Notice that $u_4, u_7, u_{10}$ are swap vertices of $P$, and for each fixed $i\in\{1,2\}$, $A \cup \{s_i, z\}$ contains a copy of~$P$ with startpoint vertex~$v$ and endpoint~$v'$ in which vertex~$z$ plays the role of~$u_{10}$, and either~$s_1$ plays the role of~$u_4$ or~$s_2$ plays the role of~$u_7$.
 }
 \label{fig::local_abs_total_deg}
\end{figure}
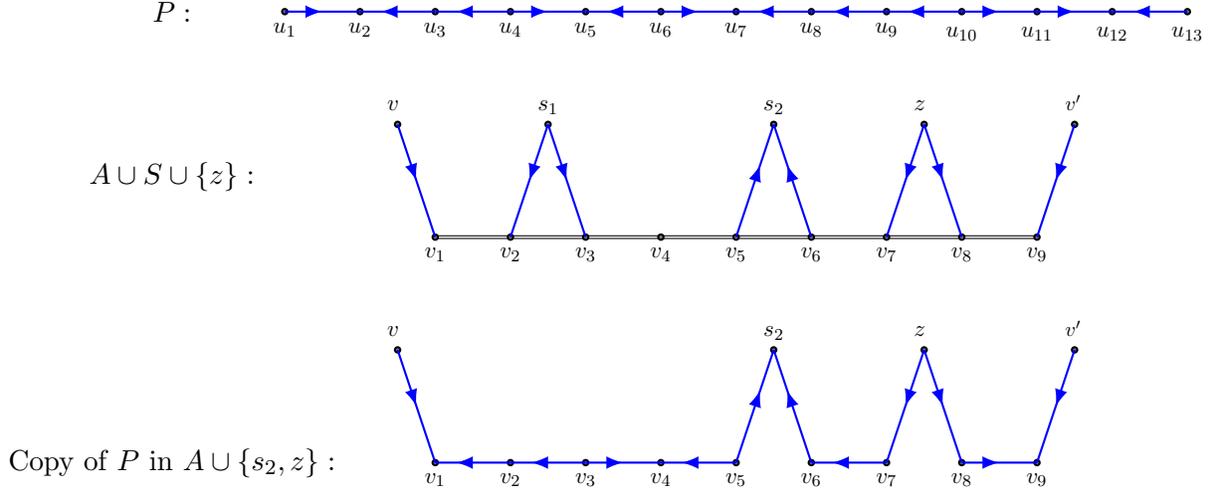

Our global absorber is structured and operates similarly to the global absorber in Section~\ref{sec::absorber}: given a set $R$ of vertices we wish to absorb, we first absorb $R$ using some vertices from a specific set $X$ of vertices, whose properties are given in Definition~\ref{def::total:reservoir}; the rest of the global absorber absorbs what is left of $X$ using carefully constructed local absorbers. The existence of an appropriate $X$ is given by Lemma~\ref{lem::total:reservoir}. Lemma~\ref{lem::total:link} helps Lemma~\ref{lem::total:chain} to absorb $R$ using $X$, and Lemma~\ref{lem::total:global} is where we actually construct the global absorber.

First, we need the following simple observation.

\begin{fact}\label{prop::manygoodout}
Let $n \in \mathbb N$ and $\alpha >0$ such that $2\alpha n+1 \leq n$.
Let $D$ be an $n$-vertex digraph with $\delta(D) \geq 2\alpha n$. Then there exists a partition $V^+ \cup V^-$ of $V(D)$ such that
for  each $* \in \{+,-\}$ we have that
$|V^*| \geq \alpha n/2$ and $d^*(v) \geq \alpha n/2$ for every $v \in V^*$.
\end{fact}

\begin{proof}
Let $U^* := \{v \in V(D) : d^*(v) \geq \alpha n/2\}$ for each $* \in \{+,-\}$. Since $\delta(D) \geq 2\alpha n$, we have that
\[ \frac{\alpha n}{2} (n-|U^*|) + n|U^*| \geq |E(D)| \geq \alpha n^2 ,\]
which yields $|U^*| \geq \alpha n/2$ for each $* \in \{+,-\}$. Since $\delta(D) \geq 2\alpha n$, $U^+\cup U^-= V(D)$.

If $|U^+ \setminus U^-| \geq \alpha n/2$, then $V^+ := U^+ \setminus U^-$ and $V^- := U^-$ is a desired partition of $V(D)$. Similarly, we get the desired partition if $|U^- \setminus U^+| \geq \alpha n/2$.

Otherwise, we must have that  $|U^+ \cap U^-|\geq n-\alpha n \geq \alpha n +1$. In this case,
we partition $U^+ \cap U^-$ into $A \cup B$ with $||A| - |B|| \leq 1$ and set $V^+ := (U^+ \setminus U^-) \cup A$ and $V^- := (U^- \setminus U^+) \cup B$.
\end{proof}

\begin{definition}\label{def::total:reservoir}
Let $n \in \mathbb N$ and $\alpha, \beta >0$.
Let $D$ be an $n$-vertex digraph. We call $X \subseteq V(D)$ an $(\alpha,\beta,m)$-reservoir if
\begin{itemize}
    \item $|X| = (1+\beta)m$, where $m$ and $\beta m$ are both integers;
    \item for every $v \in V(D)$ and for every $* \in \{+,-\}$, if $d^*(v) \geq \alpha n / 2$, then $|N^*(v) \cap X| \geq 2\beta m$;
    \item there is a partition $X^+ $, $ X^-$ of $X$ such that $|X^+|, |X^-| \geq 2\beta m$, and for each $* \in \{+,-\}$
    we have that 
    $d^*(v) \geq \alpha n/2$ for every $v \in X^*$.
\end{itemize}
\end{definition}

\begin{lemma}\label{lem::total:reservoir}
Let $\alpha, \beta >0$ such that $2\beta \leq \alpha/3 \leq 1/9$, and let $m,n \in \mathbb N$ such that $\beta m \in \mathbb N$, $n$ is sufficiently large, and $\log n \ll m \leq 0.9 n$. Let $D$ be an $n$-vertex digraph with $\delta(D) \geq 2\alpha n$. Then there exists an $(\alpha,\beta,m)$-reservoir in $D$.
\end{lemma}

\begin{proof}
By Fact~\ref{prop::manygoodout}, there exists a partition $V^+, V^-$ of $V(D)$ such that 
for each
$* \in \{+,-\}$ we have that
$|V^*| \geq \alpha n/2$ and 
$d^*(v) \geq \alpha n/2$ for every $v \in V^*$.

Let $X $ be a randomly selected subset of $V(D)$ of size $(1+\beta)m$. 
Set $X^+ := V^+\cap X$ and $ X^-:=  V^-\cap X$. 
Then by the Chernoff bound for the hypergeometric distribution, with positive probability the following hold: for every $v \in V(D)$ and for each $* \in \{+,-\}$, $d^*(v) \geq \alpha n/2$ implies $|N^*(v) \cap X| \geq \alpha m / 3 \geq 2 \beta m$;  $|X^+ |, |X^-| \geq 2 \beta m$.
 Thus, $X$ is an $(\alpha,\beta,m)$-reservoir, as desired.
\end{proof}

\begin{lemma}\label{lem::total:link}
Let $\alpha, \beta, \ep >0$ and $k, m, n \in \mathbb N$ so that  $\beta m/6 \geq 2 \ep n$ and $4 \leq k \leq \frac{3}{2} \beta m$.
Let $P$ be an oriented path on $k$ vertices. 
Let $D$ be an $n$-vertex $\ep n$-bipseudorandom digraph with $\delta(D) \geq 2\alpha n$ such that $D$ has an $(\alpha,\beta,m)$-reservoir $X$. For every distinct $v,v' \in V(D) \setminus X$ such that $(v,v')$ is $(\alpha/2)$-compatible with $P$, and for every $U \subseteq X$ with $|U| + k \leq \frac{3}{2} \beta m$, there exists a copy of $P$ in $D[(X \setminus U) \cup \{v,v'\}]$ with startpoint $v$ and endpoint $v'$.
\end{lemma}

\begin{proof}
Let $X^+$, $X^-$ be the partition of $X$ as given in Definition~\ref{def::total:reservoir}. Let $P = (u_1, \dots, u_k)$. Fix an arbitrary $U \subseteq X$ with $|U| + k \leq \frac{3}{2} \beta m$, and let $v_1, v_k \in V(D) \setminus X$ be such that $(v_1, v_k)$ is $(\alpha/2)$-compatible with $P$. We will construct a copy $Q = (v_1, \dots, v_k)$ of $P$ in $D[(X \setminus U) \cup \{v_1, v_k\}]$ in stages, in all but the final step adding  two vertices at a time.

For some $i \leq (k-2)/2$, assume that there is a copy $Q^{\leq 2i-1} = (v_1, \dots, v_{2i-1})$ of $(u_1, \dots, u_{2i-1})$ in $D[(X\setminus U) \cup \{v_1\}]$ such that $d^*(v_{2i-1}) \geq \alpha n/2$, where $* = \sigma(u_{2i-1} u_{2i})$. Note that $Q^{\leq 1}: = (v_1)$ satisfies this for $i=1$. Let $B_1 := N^*(v_{2i-1}) \cap X$ for $* = \sigma(u_{2i-1} u_{2i})$ and $B_2 := X^*$ for $* = \sigma(u_{2i+1} u_{2i+2})$. Since $|B_1|, |B_2| \geq 2\beta m$ and $|U \cup V(Q^{\leq 2i-1})| \leq |U| + k \leq \frac{3}{2} \beta m$, there exist disjoint subsets $B_i' \subseteq B_i$ of size at least $\beta m/4 \geq \ep n$ disjoint from $U$ and $V(Q^{\leq 2i-1})$. Since $D$ is $\ep n$-bipseudorandom, there exists $v_{2i} \in B_1'$ and $v_{2i+1} \in B_2'$ such that $(v_{2i-1}, v_{2i}, v_{2i+1})$ is a copy of $(u_{2i-1}, u_{2i}, u_{2i+1})$. We thus obtain $Q^{\leq 2i+1} := Q^{\leq 2i-1} \circ (v_{2i-1}, v_{2i}, v_{2i+1})$ as a copy of $(u_1, \dots, u_{2i+1})$ in $D[(X \setminus U) \cup \{v_1\}]$ with $d^*(v_{2i+1}) \geq \alpha n/2$ for $* = \sigma(u_{2i+1} u_{2i+2})$.

If $k$ is even, then we slightly modify the last step, after constructing  $Q^{\leq k-3}$. Similarly as before, we can find a bidirected edge $\overleftrightarrow{v_{k-2} v_{k-1}}$ between $N^*(v_{k-3}) \cap X$ with $* = \sigma(u_{k-3} u_{k-2})$ and $N^*(v_k) \cap X$ with $* = \sigma(u_k u_{k-1})$ disjoint from $U \cup V(Q^{\leq k-3})$. Thus $Q := Q^{\leq k-3} \circ (v_{k-3}, v_{k-2}, v_{k-1}, v_k)$ contains a copy of $P$ with startpoint $v_1$, endpoint $v_k$, and all internal vertices in $X \setminus U$.

If $k$ is odd, then we construct $Q^{\leq k-4}$ and use Lemma~\ref{lem::connecting} in place of the pseudorandom property. Let $B_1 := N^*(v_{k-4}) \cap X$ for $* = \sigma(u_{k-4} u_{k-3})$, $B_2: = X$, and $B_3: = N^*(v_k) \cap X$ with $* = \sigma(u_k u_{k-1})$. Since $|B_1|, |B_2|, |B_3| \geq 2\beta m$ and $|U \cup V(Q^{\leq k-4})| \leq \frac{3}{2} \beta m$, there exists disjoint subsets $B_i' \subseteq B_i$ of size at least $\beta m/6 \geq 2\ep n$ disjoint from $U$ and $V(Q^{\leq k-4})$. By Lemma~\ref{lem::connecting}, there exists a bidirected path $(v_{k-3}, v_{k-2}, v_{k-1})$ with $v_{k-3} \in B_1$, $v_{k-2} \in B_2$, and $v_{k-1} \in B_3$. Thus $Q := Q^{\leq k-4} \circ (v_{k-4}, v_{k-3}, v_{k-2}, v_{k-1}, v_k)$ contains a copy of $P$ with startpoint $v_1$, endpoint $v_k$, and all internal vertices in $X \setminus U$.
\end{proof}

\begin{lemma}\label{lem::total:chain}
Let $\alpha, \beta , \ep >0$ and  $m, n \in \mathbb N$ so that $\beta m/6 \geq 2 \ep n$. Let $D$ be an $n$-vertex $\ep n$-bipseudorandom digraph with $\delta(D) \geq 2\alpha n$, and suppose that $D$ has an $(\alpha,\beta,m)$-reservoir $X$. Let $R \subseteq V(D) \setminus X$ so that $|R|\geq 2$, and let $v, v' \in R$ be distinct. Let $P$ be an oriented path  on $|R| + \beta m$ vertices containing at least $4|R|-6$ swap vertices. If $(v,v')$ is $(\alpha/2)$-compatible with $P$, then there exists a copy of $P$ in $D[R \cup X]$ with startpoint $v$ and endpoint $v'$ that covers $R$.
\end{lemma}

Note that if $P$ has $|R| + \beta m$ vertices and at least $4|R|-6$ swap vertices, then $|R| \leq \frac{1}{3} (\beta m + 4)$. This implies that $P$ has less than $\frac{3}{2}\beta m$ vertices, which allows us to use Lemma~\ref{lem::total:link}. 

\begin{proof}
Label $R$ as $\{v_1, \dots, v_\ell\}$, with $v =: v_1$ and $v' =: v_\ell$. Set
 $k := \beta m + \ell$ and write $P = (u_1, \dots, u_k)$. So $P$ contains at least $4\ell-6$ swap vertices and  $(v_1, v_\ell)$ is $(\alpha/2)$-compatible with $P$.

Choose $\ell-2$ swap vertices of $P$, $u_{i_2}, \dots, u_{i_{\ell-1}}$, such that
\begin{itemize}
    \item $i_{j+1} - i_j \geq 3$ for $j \in [\ell-1]$, where $i_1 := 1$ and $i_{\ell}: = k$,
    \item $d^*(v_j) \geq \alpha n$ with $* = \sigma(u_{i_j} u_{i_j+1})$, for every $2 \leq j \leq \ell-1$.
\end{itemize}
This is possible because $P$ has at least $4\ell-6$ swap vertices, and they alternate having in- or outdegree $2$. Let $P_j := (u_{i_j}, \dots, u_{i_{j+1}})$ for $j \in [\ell-1]$, and observe that $(v_j, v_{j+1})$ is $(\alpha/2)$-compatible with $P_j$. Since the number of vertices used in total is at most $|P| \leq \frac{3}{2} \beta m$, by Lemma~\ref{lem::total:link}, we can find pairwise internally disjoint copies $Q_j$ of $P_j$ in $D[X \cup \{v_j, v_{j+1}\}]$ with startpoint $v_j$ and endpoint $v_{j+1}$. Concatenating the $Q_j$ as $Q := Q_1 \circ \cdots \circ Q_{\ell-1}$, we obtain a copy of $P$ in $D[R \cup X]$ with startpoint $v_1$ and endpoint $v_{\ell}$ covering $R$; see Figure~\ref{fig::total:chain}.
\end{proof}

\begin{figure}[h!]
\centering
\begin{tikzpicture}[thick,scale=1]

\node[draw=none, fill=white] at (0.7,2.25) {$Q_1$}; 
\node[draw=none, fill=white] at (4,2.5) {$Q_2$}; 
 
\begin{scope}[xshift = -1cm, yshift = 0cm]
\node[draw=none, fill=white] at (12,2.5) {$X^-$};
\foreach \y in {1,4} {
    \draw (\y,2.25) node [] {};}
\draw[rounded corners=15] (-.5,1.5) rectangle (11.5, 3);
\end{scope} 

\begin{scope}[xshift = -1cm, yshift = 1.5cm]
\node[draw=none, fill=white] at (12,2.5) {$X^+$};
\foreach \y in {2,3,5,6} {
    \draw (\y,2.25) node [] {};}
\draw[rounded corners=15] (-.5,1.5) rectangle (11.5, 3);
\end{scope} 

\begin{scope}[xshift = -1cm, yshift = -2cm]
\node[draw=none, fill=white] at (12,2.5) {$R$};

\draw (0.5,2.25) node [label=below:{$v_{1}$}] {};
\draw (2.5,2.25) node [label=below:{$v_{2}$}] {};
\draw (6.5,2.25) node [label=below:{$v_{3}$}] {};
\node[draw=none, fill=white] at (8.5,2.25) {$\ldots$};
\draw (10.5,2.25) node [label=below:{$v_{\ell}$}] {};
    
\draw[rounded corners=15] (-.5,1.5) rectangle (11.5, 3);
\end{scope}

\begin{scope}[yshift = .25cm]
\draw[middlearrow={latex}, blue] (-0.5,0) -- (0,2);
\draw[thin, double] (0,2) -- (1,3.5);
\draw[middlearrow={latex}, blue] (1,3.5) -- (1.5,0);
\draw[middlearrow={latex reversed}, blue] (1.5,0) -- (2,3.5);
\draw[thin, double] (2,3.5) -- (3,2);
\draw[middlearrow={latex reversed}, blue] (3,2) -- (4,3.5);
\draw[thin, double] (4,3.5) -- (5,3.5);
\draw[middlearrow={latex}, blue] (5,3.5) -- (5.5,0);
\draw[blue] (5.5,0) -- (6,1);
\draw[blue] (9.5,0) -- (9,1);
\end{scope}

\end{tikzpicture}
\caption{
A copy of $P$ in $D[R\cup X]$ with startpoint $v_1$ and endpoint $v_\ell$ that covers all the vertices in $R$, as found in Lemma~\ref{lem::total:chain}. The path between $v_i$ and $v_{i+1}$ is found by Lemma~\ref{lem::total:link}. The double edges indicate that both orientations are present, as found by applying the bipseudorandom property.}
\label{fig::total:chain}
\end{figure}
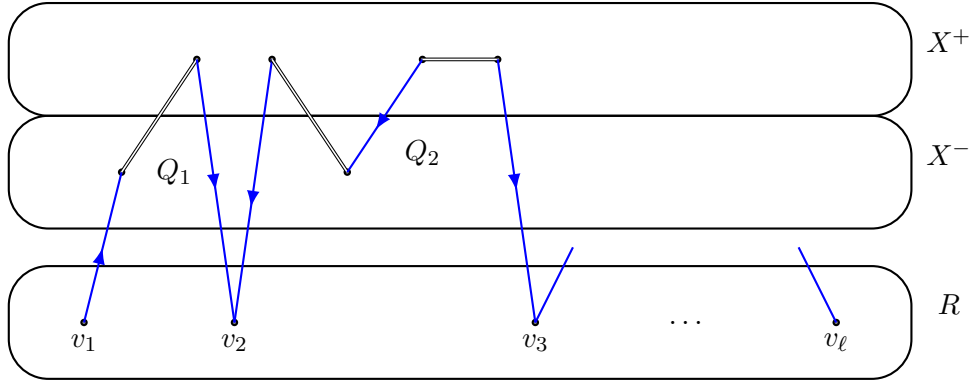

Before proving the main absorbing lemma, we need  a lemma which allows us to construct long paths with endpoints that are compatible with a given short path. This is useful in the construction of the global absorber and in the application of the global absorber in Section~\ref{sec::totalproof}.

\begin{lemma}\label{lem::total:longpath}
Let $n \in \mathbb N$ and $0<\ep, \alpha <1/3$ such that
${1}/{n} \leq \ep \leq \alpha/32$. Let $D$ be an $n$-vertex $\ep n$-bipseudorandom digraph with $\delta(D) \geq 2\alpha n$, and let $U \subseteq V(D)$ with $|U| \leq \alpha n/4$. For every $2 \leq k \leq (1-8\ep)n - |U|$ and for every $(*_1, *_2) \in \{+,-\}^2$, there exists a bidirected path on $k$ vertices in $D \setminus U$ with startpoint $v$ and endpoint $v'$ satisfying $d^{*_1}_D(v), d^{*_2}_D(v') \geq \alpha n/2$.
\end{lemma} 
% AT: added $0<\ep, \alpha <1/3$ condition here so that we can apply Fact~\ref{prop::manygoodout}
%% AT: it is important that we put the subscript D in the degrees here... otherwise it could look like we mean degree in D\U
%% IA: Constants in this lemma are not optimal (e.g. 8\ep could be 6\ep and \alpha/28 could be \alpha/24) but this do not change dependency of C in \alpha.

\begin{proof}
Fix $k \leq (1-8\ep)n - |U|$, $*_1$, and $*_2$. By Fact~\ref{prop::manygoodout}, and as $|U| \leq \alpha n/4$, we can partition $V(D) \setminus U$ as $V^+ \cup V^-$, where 
for each $* \in \{+,-\}$ we have that
$|V^*| \geq \alpha n/4$, and  $d^*_D(v) \geq \alpha n/2$ for every  $v \in V^*$. Lemma~\ref{lem::longpath} gives a bidirected path $Q^*$ in $D[V^*]$ on at least $|V^*| - 2\ep n$ vertices for each $* \in \{+,-\}$.

\textbf{Case 1:} $*_1 \neq *_2$. We take the last $\ep n$ vertices of $Q^-$ and the first $\ep n$ vertices of $Q^+$ and find a bidirected edge between them, which exists since $D$ is $\ep n$-bipseudorandom. This gives a bidirected path $Q$ on at least $|V(Q^-)| + |V(Q^+)| - 2 \ep n \geq n - |U| -6\ep n$ vertices. Truncating $Q$ at both ends, we obtain a bidirected path on $k$ vertices with startpoint in $V^{*_1}$ and endpoint in $V^{*_2}$.

\textbf{Case 2:} $*_1 = *_2$. Without loss of generality, assume $*_1 = *_2 = +$. If $k \leq |V(Q^+)|$, simply truncate $Q^+$ to $k$ vertices to obtain the desired path. If $k > |V(Q^+)|$, then truncate $Q^-$ to $Q^-_1$ on $k-|V(Q^+)|+ 4\ep n \geq 2\ep n$ vertices, which is possible because
\[ k - |V(Q^+)| + 4\ep n \leq (1-8\ep)n - |U| - (|V^+|-2\ep n) + 4\ep n = |V^-| - 2 \ep n \leq |V(Q^-)| .\]
Between the first $\ep n$ vertices of $Q^-_1$ and the first $\ep n$ vertices of $Q^+$ we find a bidirected edge, as well as between the last $\ep n$ vertices of $Q^-_1$ and the `second' $\ep n$ vertices of $Q^+$. This yields a bidirected path $Q$ on at least $k$ vertices and at most $k + 4\ep n$ vertices with startpoint and endpoint in $V^+$. Since
\[ |V(Q^+)| \geq |V^+| - 2\ep n \geq \alpha n/4 - 2\ep n \geq 6\ep n ,\]
we may truncate $Q$ to $k$ vertices and still have the startpoint and endpoint in $V^+$.
%AT: I changed some numbers in proof so that it is easier for the reader to see what is going on
\end{proof}

We are now ready to prove the absorbing lemma for Theorem~\ref{thm:main}.
\begin{lemma}[Absorbing lemma]\label{lem::total:global}
For every $0<\alpha, \eta \ll 1$, and every $0<\eps \leq \alpha^2 \eta^4/10^8$, there exists an 
$n_0 \in \mathbb N$ such that for all $n \geq n_0$  the following holds.
%Hierarchies with $\ll$ in have to be defined from right to left; this is why I have edited in this way. 
Let $D$ be an $n$-vertex $\ep n$-bipseudorandom digraph with $\delta(D) \geq 2\alpha n$. Let $P$ be an oriented path on $\ceil{\alpha n/4}$ vertices with at least $\eta(|V(P)|-2)$ swap vertices. Then $D$ contains a $(P,\alpha/2)$-global absorber of size at most $\ceil{\alpha n/4} - 9\ep n$.
\end{lemma}

\begin{proof}
Let $n$ be chosen sufficiently large so that all our calculations will hold. Let $D$ and $P$ be as in the statement of the lemma.
We first construct the global absorber and then prove that it is indeed a $(P,\alpha/2)$-global absorber. Define $\beta, m >0$ such that $\alpha/7 \leq \beta \leq \alpha /6$ and $\frac{\alpha \eta^3}{50000} n \leq m \leq \frac{\alpha \eta^3}{49999} n$, and so that $m$ and $\beta m$ are integers. 
Without loss of generality we may assume that 
$516/\eta$ is an integer. 

Let $X$ be an $(\alpha,\beta,m)$-reservoir, whose existence is  guaranteed by Lemma~\ref{lem::total:reservoir}. Let $Y, Z \subseteq V(D) \setminus X$ be disjoint sets of $2m$ and $3m$ vertices, respectively. Fix an  auxiliary graph $H_m$ isomorphic to the graph from Lemma~\ref{lem::Mont-graph} on $X \cup Y \cup Z$. For each $z \in Z$, we set $N_z := N_{H_m}(z)$.

We split $P$ into several pieces as follows. Let $P = (u_1, \dots, u_k)$, with $k: = \ceil{\alpha n/4}$. Set $r: = 9 \ep n$ and $\ell := \beta m + r - 4$. Define
\begin{itemize}
    \item $P^0 := (u_1, u_2, u_3, u_4)$,
    \item $P^1 := (u_4, \dots, u_{j+1})$,
    \item $P^2 := (u_{j+1}, \dots, u_{j+\ell})$,
    \item $P^3 := (u_{j+\ell}, \dots, u_{k-3})$,
    \item $P^4 := (u_{k-3}, u_{k-2}, u_{k-1}, u_k)$,
\end{itemize}
where $4 \leq j \leq k-4-\ell$ is chosen so that $P^2$ has at least $4r - 6$ swap vertices.
This is possible because otherwise there are at most
\begin{align}\label{usethis}
 \left\lfloor\frac{k-10}{\ell}\right\rfloor (4r-5) + (\ell -1) +8 < 12 \frac{\alpha n}{\beta m} \ep n + \beta m + 9\ep n < \eta(k-2)n 
 \end{align}
swap vertices of $P$, a contradiction.
To see where (\ref{usethis}) comes from, one should consider a partition of $V(P)\setminus \{u_1,\dots, u_5, u_{k-4},\dots , u_k\}$ into $\left\lfloor\frac{k-10}{\ell}\right\rfloor$ sets of $\ell$ consecutive vertices along $P$, and one `leftover' set of size at most $\ell -1$. One should also note that it is only the internal vertices along a path that can be swap vertices; this is why (\ref{usethis}) has a $4r-5$ term rather than a $4r-7$ term.
% 12*350000 \alpha \ep / \alpha^2 \eta^3 \leq \eta \alpha /4 so \ep = \alpha^2 \eta^4 / 4*12*350000

We claim that  for some $i=1,3$ we have that $P^i$ has at least $6192m/\eta^2$ vertices, and at least  $\eta|V(P^i)|/2$ of those vertices are swap vertices of $P^i$. This is because if $P^1$ has fewer than $6192m/\eta^2$ vertices, then $P^3$ must have at least
\[ \eta (k-2) - 10 - \ell - 6192m/\eta^2 \geq \max\left\{ \frac{\eta}{2} k, 6192m/\eta^2 \right\} \]
% 9 \ep n \leq \beta m implies \ell \leq 2\beta m
% \eta k / 2 = \eta \alpha n / 8 >> 6100m/\eta^2
swap vertices; if $P^1$ has less than an $\eta/2$-proportion of its vertices being swap vertices, then $P^3$ must have at least
\[ \eta (k-2) - 10 - \ell - \frac{\eta}{2} j \geq \max\left\{\frac{\eta}{2} |V(P^3)|, 6192m/\eta^2\right\} \]
% \eta k / 2 >> \ell (much easier than the above)
% \eta k >> 6099 m/\eta^2 + \eta j/2 (same as above)
swap vertices. 
Assume that our claim held for $i=1$; the $i=3$ case follows by an 
 essentially identical argument with $P^3$ in place of $P^1$.

Our local absorbers will be `housed' in $P^1$; the segment $P^2$ will be used to absorb  $R$ from Definition~\ref{def::total:global}; $P^0$ and $P^4$ are used to ensure the copy of $P$ we find has the correct startpoint and endpoint; $P^3$ is simply used to fill up the remaining part of $P$. 

\smallskip

Let $p := 516/\eta$ (recalling that $p$ is an integer) and $s := \floor{\frac{j-8}{p}}-1$. 
For $0 \leq i \leq s$, define~$P^1_i := (u_{ip+5}, \dots, u_{(i+1)p+4})$; see Figure~\ref{fig::path_decomp}. We call $P^1_i$ \emph{good} if it contains at least $3 \cdot 40 + 7 = 127$ swap vertices of $P^1_i$; this will be enough to apply Lemma~\ref{lem::total:local} later to find a local absorber for $(N_z,z)$. 

Note that there must be at least $3m$ good $P^1_i$, since otherwise $P^1$ has at most
\begin{align}\label{usethis2}
(s+1) \cdot 128 + (3m-1) \cdot p + (j-2-(s+1)p) < 129 (s+1) + 3mp  \leq \frac{\eta}{2} |V(P^1)|
\end{align}
swap vertices,  a contradiction.
Note that the first term in (\ref{usethis2}) is 
$(s+1) \cdot 128$ as, including the startpoint and endpoint of $P^1_i$, 
$V(P^1_i)$ may contain at most $128$ swap vertices of $P^1$, and yet only contain at most $126$ swap vertices of $P^1_i$. The second term in (\ref{usethis2}) corresponds to the  $P^1_i$s in which every vertex may be 
a swap vertex of $P^1$. The third term in (\ref{usethis2}) counts all those vertices on $P^1$ that do not 
live in one of the $P^1_i$. The final inequality in 
(\ref{usethis2}) follows as
 $\frac{\eta}{4} |V(P^1)| \geq 3mp$ (as  $|V(P^1)|\geq 6192m/\eta ^2$ and $p= 516/\eta$), and as $129(s+1) \leq 129 \frac{j-2}{p} = \frac{\eta}{4} |V(P^1)|$ (by the definition of $s$ and $p$).

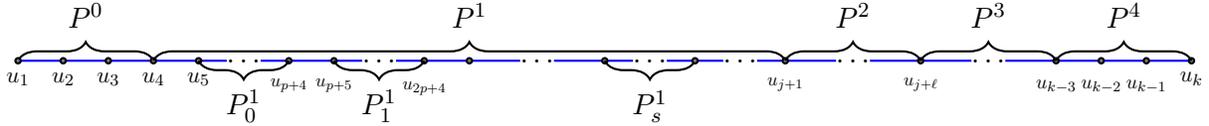
\begin{figure}[h!]
\centering
\begin{tikzpicture}[thick,scale=0.6]
\draw[blue] (1,3) -- (27,3);
 
%P^0 
\node[draw=none, fill=white] at (2.5,4) {$P^0$};
\draw [decoration={brace, amplitude=0.25cm}, decorate] (1,3) -- (4,3);
\foreach \y in {1,...,4} {
\draw (\y,3) node [label=below:{$\scaleto{u_{\y}}{5pt}$}] {};
%\draw[blue] (\y,3) -- (\y+1,3);
}

%P_0^1
\node[draw=none, fill=white] at (6,2) {$P_0^1$};
\draw (5,3) node [label=below:{$\scaleto{u_{5}}{5pt}$}] {};
\draw (7,3) node [label=below:{$\scaleto{u_{p+4}}{5pt}$}] {};
\node[draw=none, fill=white] at (6,3) {$\ldots$};
\draw [decoration={brace, amplitude=0.25cm}, decorate] (7,3) -- (5,3);

%P_1^1
\node[draw=none, fill=white] at (9,2) {$P_1^1$};
\draw (8,3) node [label=below:{$\scaleto{u_{p+5}}{5pt}$}] {};
\draw (10,3) node [label=below:{$\scaleto{u_{2p+4}}{5pt}$}] {};
\node[draw=none, fill=white] at (9,3) {$\ldots$};
\draw [decoration={brace, amplitude=0.25cm}, decorate] (10,3) -- (8,3);

%P_s^1
\draw (11,3) node [] {};
\node[draw=none, fill=white] at (12.5,3) {$\ldots$};
\draw (14,3) node [] {};
\node[draw=none, fill=white] at (15,3) {$\ldots$};
\node[draw=none, fill=white] at (15,2) {$P_s^1$};
\draw [decoration={brace, amplitude=0.25cm}, decorate] (16,3) -- (14,3);
\draw (16,3) node [] {};
\node[draw=none, fill=white] at (17,3) {$\ldots$};

%P^1
\node[draw=none, fill=white] at (11,4) {$P^1$};
\draw [decoration={brace, amplitude=0.25cm}, decorate] (4,3) -- (18,3);

%P^2
\node[draw=none, fill=white] at (19.5,4) {$P^2$};
\draw (18,3) node [label=below:{$\scaleto{u_{j+1}}{5pt}$}] {};
\node[draw=none, fill=white] at (19.5,3) {$\ldots$};
\draw (21,3) node [label=below:{$\scaleto{u_{j+\ell}}{5pt}$}] {};
\draw [decoration={brace, amplitude=0.25cm}, decorate] (18,3) -- (21,3);

%P^3
\node[draw=none, fill=white] at (22.5,4) {$P^3$};
\node[draw=none, fill=white] at (22.5,3) {$\ldots$};
\draw (24,3) node [label=below:{$\scaleto{u_{k-3}}{5pt}$}] {};
\draw [decoration={brace, amplitude=0.25cm}, decorate] (21,3) -- (24,3);

%P^4
\node[draw=none, fill=white] at (25.5,4) {$P^4$};
\draw (25,3) node [label=below:{$\scaleto{u_{k-2}}{5pt}$}] {};
\draw (26,3) node [label=below:{$\scaleto{u_{k-1}}{5pt}$}] {};
\draw (27,3) node [label=below:{$\scaleto{u_{k}}{5pt}$}] {};
\draw [decoration={brace, amplitude=0.25cm}, decorate] (24,3) -- (27,3);

\end{tikzpicture}
\caption{
The pieces~$P^i$ which compose the path~$P$ as in the proof of Lemma~\ref{lem::total:global}. $P^1$ has at least~$6192m/\eta^2$~vertices and at least an~$\eta/2$-proportion of those vertices are swap vertices, 
while~$P^2$ has at least~$4r -6$~swap vertices. We further divide $P^1$ into disjoint $P^1_i$ of equal length
which house the local absorbers.}
\label{fig::path_decomp}
\end{figure}

As $|X\cup Y\cup Z|\leq \alpha n/6$, Fact~\ref{prop::manygoodout} ensures
a partition
 $V^+ , V^-$   of $V(D) \setminus (X \cup Y \cup Z)$ such that $|V^*| \geq {\alpha}n/3$ and 
 $d^*_D(v) \geq \alpha n /2$ for every $v \in V^*$ and $* \in \{+,-\}$. Since $D$ is $\ep n$-bipseudorandom, 
 we can find pairwise disjoint bidirected edges $\overleftrightarrow{w_i w_i'}$ for $0 \leq i \leq s+1$, 
 where $w_i \in V^*$ with $* = \sigma(u_{ip+4} u_{ip+3})$ and $w_i' \in V^*$ with $* = \sigma(u_{ip+5} u_{ip+6})$. 
 In this way, $(w_i', w_{i+1})$ is $(\alpha/2)$-compatible with $P^1_i$ for $0 \leq i \leq s$. 

As there are at least $3m$ good $P^1_i$, we
can  assign to  each $z \in Z$ a distinct $i_z$ such that $P^1_{i_z}$ is good. 
We construct pairwise disjoint $P^1_{i_z}$-absorbers $(A_{i_z}, w_{i_z}', w_{i_z+1})$ for $(N_z,z)$, disjoint from $X \cup Y \cup Z$, via repeated applications of Lemma~\ref{lem::total:local}; this is possible  as $\ep \leq \frac{\alpha}{8p}$, and  because in the  process of constructing these local absorbers we use  at most $3mp\leq \alpha n/4$ vertices in total.

Let $I := \{ i : 0 \leq i \leq s, \nexists z \in Z \text{ with } i = i_z\}$. For each $i \in I$, we find a copy $Q^1_i$ of $P^1_i$ with startpoint $w_i'$ and endpoint $w_{i+1}$ such that they are pairwise disjoint and disjoint from $X \cup Y \cup Z$ and $A_{i_z}$ for all $z \in Z$. This is achieved by applying Lemma~\ref{lem::connecting} as follows. Let $B_2 := N^*_D(w_i')$ for $* = \sigma(u_{ip+5} u_{ip+6})$ and $B_{p-1} := N^*_D(w_{i+1})$ for $* = \sigma(u_{(i+1)p+4} u_{(i+1)p+3})$, and note that $|B_2|, |B_{p-1}| \geq \alpha n/2$. Let $B_2' \subseteq B_2$ and $B_{p-1}' \subseteq B_{p-1}$ be disjoint from each other and all the vertices in $X \cup Y \cup Z$, $\bigcup_{z\in Z} A_{i_z}$, and the other $Q^1_i$, of which there are at most $k = \ceil{\alpha n/4}$, so that $|B_2'|, |B_{p-1}'| \geq \alpha n / 10 \geq 2\ep n$. Let $B_3', \dots, B_{p-2}'$ be arbitrary pairwise disjoint subsets of $V(D)$ of size at least $2\ep n$, disjoint from everything from before. By Lemma~\ref{lem::connecting} there is a bidirected path through the $B_i'$; adding $w_i'$ and $w_{i+1}$ on either end gives the desired $Q^1_i$.

Finally for $P^1$, we find a bidirected edge $\overleftrightarrow{w_{s+2} w_{s+2}'}$ disjoint from $X \cup Y \cup Z$, the $Q^1_i$ for $i \in I$, and the local absorbers $A_{i_z}$, with the property that $w_{s+2} \in V^*$ with $* = \sigma(u_{j} u_{j-1})$ and $w_{s+2}' \in V^*$ with $* = \sigma(u_{j+1} u_{j+2})$. In a similar fashion as before, since $(w_{s+1}', w_{s+2})$ is $(\alpha/2)$-compatible with $(u_{(s+1)p+5}, \dots, u_j)$, we find a copy $Q^1_{s+1}$ of $(u_{(s+1)p+5}, \dots, u_j)$ with startpoint $w_{s+1}'$ and endpoint $w_{s+2}$.

Let
\[ A^1 := \{w_0\} \cup \bigcup_{z \in Z} A_{i_z} \cup \bigcup_{i\in I} V(Q^1_i) \cup V(Q^1_{s+1}) \cup \{w_{s+2}'\}, \]
then
\begin{equation}\label{eq::A1size}
|A^1| = 1 + 3m(p-2) + (s+1-3m)p + (j-(s+1)p-4) + 1 = j - 6m - 2 =|V(P^1)|-6m.
\end{equation}

\begin{claim}\label{claim::TRIGGERED}
Given any matching in $H_m$ covering $Z$, let $Z' \subseteq X \cup Y$ denote  the set of vertices matched to $Z$. Then there exists a copy of $P^1$ in $D[A^1 \cup Z \cup Z']$ with startpoint $w_0$ and endpoint $w_{s+2}'$.
\end{claim}

\begin{proof}
For each $z \in Z$, we activate the local absorber $A_{i_z}$ for $z$ and its matched vertex $z'$ in $Z'$, yielding a copy $Q^1_{i_z}$ of $P^1_{i_z}$ containing $z$ and $z'$. We concatenate the $Q^1_i$ to obtain a copy of $P^1$ as
\[ (w_0, w_0') \circ Q^1_0 \circ (w_1, w_1') \circ Q^1_1 \circ \cdots \circ Q^1_{s+1} \circ (w_{s+2}, w_{s+2}') \]
with startpoint $w_0$ and endpoint $w_{s+2}'$; see Figure~\ref{fig::claim}.
\end{proof}

\begin{figure}[h!]
\centering
\begin{tikzpicture}[thick,scale=0.8]

\node[draw=none, fill=white] at (3.5,2.5) {$Q_0^1$};
\node[draw=none, fill=white] at (7.5,2.5) {$Q_1^1$};
\node[draw=none, fill=white] at (13.5,2.5) {$Q_{s}^1$};
\node[draw=none, fill=white] at (17,2.5) {$Q_{s+1}^1$};

\draw[thin, double] (1,3) -- (2,3);
\draw[blue] (2,3) -- (5,3);
\draw[thin, double] (5,3) -- (6,3);
\draw[blue] (6,3) -- (9,3);
\draw[thin, double] (9,3) -- (10,3);
\draw[thin, double] (11,3) -- (12,3);
\draw[blue] (12,3) -- (15,3);
\draw[thin, double] (15,3) -- (16,3);
\draw[blue] (16,3) -- (18,3);
\draw[thin, double] (18,3) -- (19,3);

\draw (1,3) node [label=above:{$w_{0\phantom{+0}}\phantom{'}$}] {};
\draw (2,3) node [label=above:{$w'_{0\phantom{+0}}$}] {};
\draw (5,3) node [label=above:{$w_{1\phantom{+0}}\phantom{'}$}] {};
\draw (6,3) node [label=above:{$w'_{1\phantom{+0}}$}] {};

\draw (9,3) node [] {};
\node[draw=none, fill=white] at (10.5,3) {$\ldots$};
\draw (12,3) node [] {};

\draw (15,3) node [label=above:{$w_{s+1}\phantom{'}$}] {};
\draw (16,3) node [label=above:{$w_{s+1}'$}] {};
\draw (18,3) node [label=above:{$w_{s+2}\phantom{'}$}] {};
\draw (19,3) node [label=above:{$w_{s+2}'$}] {};

\end{tikzpicture}
\caption{ A copy of $P^1$ in $D[A^1 \cup Z \cup Z']$ as found in Claim~\ref{claim::TRIGGERED}. The $w_i w_i'$ edges are double edges, as found by applying the bipseudorandom property to $V^+$ and $V^-$.}
\label{fig::claim}
\end{figure}
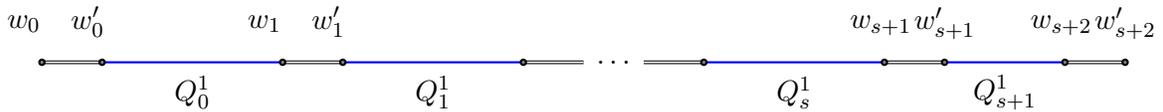

To complete the construction of the global absorber, we use Lemma~\ref{lem::total:longpath} to obtain a copy $Q^3$ of $P^3$ disjoint from $A^1 \cup X \cup Y \cup Z$ with startpoint $w \in V^*$ where $* = \sigma(u_{j+\ell} u_{j+\ell-1})$, and endpoint $w' \in V^*$ where $* = \sigma(u_{k-3} u_{k-2})$. Let
\[ A: = X \cup Y \cup Z \cup A^1 \cup V(Q^3) .\]
By~\eqref{eq::A1size}, $\ell = \beta m + r - 4$, $r = 9\ep n$, and $k = \ceil{\alpha n/4}$, we have that
\[ |A| = (6+\beta)m + (j-6m-2) + (k-2-j-\ell) = k - r = \ceil{\alpha n/4} - 9\ep n. \]

\smallskip

Now we prove that $A$ is indeed a $(P,\alpha/2)$-global absorber. Let $R \subseteq V(D) \setminus A$ with $|R| + |A| = |V(P)|$ be given, as well as $v,v' \in R$ which are $\alpha/2$-compatible with $P$. Let $R' := (R \setminus \{v,v'\}) \cup \{w_{s+2}', w\}$. By Lemma~\ref{lem::total:chain}, there exists a copy $Q^2$ of $P^2$ in $D[R' \cup X]$ covering $R'$ with startpoint $w_{s+2}'$ and endpoint $w$. By Lemma~\ref{lem::total:link}, there exists a copy $Q^0$ of $P^0$ in $D[X \cup \{v,w_0\}]$ disjoint from $V(Q^2)$ with startpoint $v$ and endpoint $w_0$; similarly, there exists a copy $Q^4$ of $P^4$ in $D[X \cup \{w', v'\}]$ disjoint from $V(Q^0)$ and $V(Q^2)$ with startpoint $w'$ and endpoint $v'$.

Let $X'$ be the set of vertices in $X$ not used in $Q^0$, $Q^2$, or $Q^4$. Thus,
\[ |X'| = |X| - (2 + |V(P_2)| - |R'| + 2) = (1+\beta)m - \ell + r - 4 = m .\]
By Lemma~\ref{lem::Mont-graph}, there exists a perfect matching between $Z$ and $X' \cup Y$. By Claim~\ref{claim::TRIGGERED}, there exists a copy $Q^1$ of $P^1$ with startpoint $w_0$ and endpoint $w_{s+2}'$ covering $X' \cup Y \cup Z \cup A^1$. Concatenate the $Q^i$ as
\[ Q := Q^0 \circ Q^1 \circ Q^2 \circ Q^3 \circ Q^4 .\]
Then $Q$ is a copy of $P$ in $D[R \cup A]$ with startpoint $v$ and endpoint $v'$. 
\end{proof}

\section{Proof of Theorem~\ref{thm:main}}\label{sec::totalproof}

\begin{proof}
To prove Theorem~\ref{thm:main} it clearly suffices to prove the case when $\eta = \alpha$ and $0< \alpha \ll 1$.
Set $\ep := \alpha^6/10^9$ and  $C := 4e/\ep^2$, and let $n \in \mathbb N$ be sufficiently large. Let $D_0$ be an $n$-vertex digraph with $\delta(D_0) \geq 2\alpha n$.

Call an oriented cycle \emph{good} if it has length between $3$ and $n$ with at most $(1-\alpha)n$ vertices of indegree $1$, and it is not a consistently oriented cycle of length $3$. 

Given any good cycle $\mathcal{C}$, we first prove that $D := D_0 \cup D^*(n,C/n)$ contains a copy of $\mathcal C$ with probability at least $1 - e^{-n}$. Note that Lemma~\ref{lem::pseudo} implies that $D^*(n,C/n)$ is $\ep n$-bipseudorandom with probability at least $1 - e^{-n}$; thus, we may assume that $D$ is an $\ep n$-bipseudorandom digraph.

When $|\mathcal{C}| = 3$, then choose an arbitrary $v \in V(D)$. Either $d^+_D(v) \geq \alpha n$ or $d^-_D(v) \geq \alpha n$. In either case, since $\ep \leq \alpha /2$, we can find an edge in the in- or out-neighborhood of $v$ using that $D$ is $\ep n$-bipseudorandom, yielding a non-consistently oriented cycle of length $3$.

When $4 \leq |\mathcal{C}| \leq (1 - 13 \ep /\alpha - 8 \ep) n$, we first fix an
$(\alpha/2, \alpha/12, \frac{12\ep}{\alpha} n)$-reservoir $X$, as given by 
Lemma~\ref{lem::total:reservoir}. In particular, $|X| \leq \frac{13\ep}{\alpha} n 
\leq \alpha n/4$, and for every $v \in V(D)$ and $* \in \{+,-\}$, if $d^*_D(v) \geq 
\alpha n/4$ then $|N^*_D(v) \cap X| \geq 2 \ep n$. Let $(u_1, u_2, u_3, u_4)$ 
be a subpath of $\mathcal{C}$. Since $|\mathcal{C}|-2 \leq (1-8\ep)n-|X|$, 
Lemma~\ref{lem::total:longpath} gives us a bidirected path $Q_\bi$ disjoint from 
$X$ on $|\mathcal{C}|-2$ vertices with startpoint $v$ and endpoint $v'$ 
satisfying $d^{*_1}_D(v), d^{*_2}_D(v') \geq \alpha n/4$ with $*_1 = \sigma(u_1
u_2)$ and $*_2 = \sigma(u_4 u_3)$. We close $Q_\bi$ into a copy of $\mathcal{C}$
by applying the bipseudorandom property to disjoint subsets of $N^{*_1}_D(v) \cap
X$ and $N^{*_2}_D(v') \cap X$.

When $|\mathcal{C}| \geq (1-13\ep/\alpha-8\ep) n$, we have at least $\alpha n - 13\ep/\alpha n - 8\ep n \geq 2\alpha n/3$ swap vertices in $\mathcal{C}$. By a standard averaging argument, there is a subpath $P$ of $\mathcal{C}$ on $\ceil{\alpha n/4}$ vertices with at least $\frac{2\alpha}{3}(|V(P)|-2)$ swap vertices. By Lemma~\ref{lem::total:global}, there exists a $(P,\alpha/2)$-global absorber $A$ of size at most $\ceil{\alpha n/4} - 9\ep n$. Let $P = (u_1, \dots, u_k)$ with $k := \ceil{\alpha n/4}$, and let $*_1 := \sigma(u_1 u_2)$ and $*_2 := \sigma(u_k u_{k-1})$. By Lemma~\ref{lem::total:longpath}, there exists a bidirected path $Q_\bi$ in $D \setminus A$ on $|\mathcal{C}| - k + 2 \leq n - |A| - 8\ep n$ vertices with startpoint $v$ and endpoint $v'$ satisfying $d^{*_1}_D(v), d^{*_2}_D(v') \geq \alpha n/2$. Let $R \subseteq V(D) \setminus (A \cup V(Q_\bi))$ be arbitrary with $|A| + |R| + 2 = k$, which exists because
\[ 0 \leq k-2-|A| = |\mathcal{C}| - |V(Q_\bi)| - |A| \leq |V(D) \setminus (A \cup V(Q_\bi))| .\]
Let $R' := R \cup \{v,v'\}$. Since $(v,v')$ is $(\alpha/2)$-compatible with $P$, $A$ is a $(P,\alpha/2)$-global absorber, and $|A| + |R'| = |V(P)|$, we have that $D[A \cup R']$ contains a copy $Q$ of $P$ with startpoint $v$ and endpoint $v'$. Joining $Q$ and $Q_\bi$ at both ends, we have a copy of $\mathcal{C}$ as desired.

Thus for every good cycle $\mathcal{C}$, we have that $D_0 \cup D^*(n,C/n)$ contains a copy of $\mathcal{C}$ with probability at least $1-e^{-n}$. Hence Lemma~\ref{lem::coupling} implies that $D_0 \cup D(n,C/n)$ contains a copy of $\mathcal{C}$ with probability at least $1-e^{-n}$. Taking a union bound over all the at most $n 2^n$ possible lengths and orientations, we have that $D_0 \cup D(n,C/n)$ contains every good oriented cycle a.a.s.

Finally, we deal with consistently oriented cycles of length $2$ and $3$. To see that $D_0 \cup D(n,C/n)$ contains a cycle of length $2$ a.a.s., we proceed as in the proof of Theorem~\ref{thm::semideg_simple}: observe that $D_0$ has at least $\alpha n^2$ directed edges, and so the probability no edge of $D(n,C/n)$ is in the opposite orientation of an edge of $D_0$ is at most
\[ (1-C/n)^{\alpha n^2} \leq e^{-C\alpha n} .\]

To see that $D_0 \cup D(n,C/n)$ contains a consistently oriented cycle of length $3$, we use the second moment method to show that a.a.s.\ there exists an edge $\forw{uv}$ of $D_0$ and a vertex $w \neq u,v$ such that $\forw{vw}, \forw{wu}$ are edges of $D(n,C/n)$. The expected number of such triangles is $\mu := e(D_0) (n-2) (C/n)^2$, while the variance is at most
\[ \mu \left( 1 + 2 \cdot 2 (n-3) (C/n) \right) \leq 5 C \mu ,\]
since such a triangle intersects $2 \cdot 2 (n-3)$ other such triangles in their random edges --- $2$ choices for which edge ($\forw{vw}$ or $\forw{wu}$) to intersect, $n-3$ choices for the other vertex $x$, and $2$ choices for whether $xw$ is the deterministic edge of the other triangle. Since $e(D_0) \geq \alpha n^2$, $\mu$ is of  order  $n$,  so by Chebyshev's inequality, $D_0 \cup D(n,C/n)$ contains a consistently oriented cycle of length $3$ a.a.s.
\end{proof}

\section{Concluding remarks}\label{sec::conc}

In this paper we have determined how many random edges one must add to a digraph with linear minimum semi-degree to a.a.s.~force all orientations of a Hamilton cycle. There has also been interest in obtaining results in the perturbed setting where the initial (di)graph is sparse.
In particular,
 Hahn-Klimroth, Maesaka, Mogge, Mohr, and Parczyk~\cite{hmmmp} proved a generalization of the graph version of Theorem~\ref{bfmthm}, where
 now $\alpha$ can be a function of $n$ that tends towards $0$ as $n$ tends to infinity.
 Let $G(n,p)$ be the \emph{binomial random graph} on the vertex set $[n]$ where every possible edge is present with probability $p$, independently of all other edges.

\begin{theorem}[Hahn-Klimroth, Maesaka, Mogge, Mohr, and Parczyk~\cite{hmmmp}]\label{hmmmp}
 Let $\alpha = \alpha(n): \mathbb N \rightarrow (0,1)$ and $C =C(\alpha) = (6+o(1))\log \frac{1}{\alpha}$. If $G_0$ is an $n$-vertex graph of minimum degree $\delta(G_0) \geq \alpha n$, then $G_0 \cup G(n,C/n)$ a.a.s.\ contains an (undirected) Hamilton cycle.
\end{theorem}

Note that one cannot take $C=o(\log \frac{1}{\alpha})$ in Theorem~\ref{hmmmp} (see~\cite[Section 1.1]{hmmmp}),
so in this sense the theorem is best possible.
We can also take $\alpha$ to be non-constant in Theorems~\ref{thm::semideg_simple} and~\ref{thm:main}; indeed (taking $\alpha =\eta$ in Theorem~\ref{thm:main}), our proofs give $C = \Theta(\alpha^{-4})$ and $C = \Theta(\alpha^{-12})$ respectively. For Theorem~\ref{thm::semideg_simple} at least,
similarly to Theorem~\ref{hmmmp}, we would expect that this could be improved 
to $C = \Theta(\log \frac{1}{\alpha})$. It would be interesting to determine the
optimal dependence of $C$ on $\alpha$.

In Theorem~\ref{thm:main} we studied randomly perturbed digraphs with linear minimum total degree. It is also natural to seek other such total degree results. For example, given any $\alpha>0$, $k\in \mathbb N$, $n \in k \mathbb N$ and any $n$-vertex digraph $D_0$ with $\delta(D_0)\geq  \alpha n$, how many random  edges must one add to $D_0$ to ensure that, a.a.s., the resulting digraph contains a $T_k$-factor? Here by \emph{$T_k$-factor} we mean a collection of vertex-disjoint transitive tournaments of size $k$ that together cover $V(D)$.

Another natural problem is to determine the number of random edges one must add to a digraph with linear minimum semi-degree to a.a.s.~force a given oriented spanning tree. The corresponding problem in the graph setting has been studied in~\cite{hmmmp, joos2, kks1}. 

{\bf Remark:} Since a version of this paper first appeared on arXiv, Morawski and Petrova~\cite{mp} have resolved this problem for fixed oriented spanning trees of bounded degree.

\smallskip

\section*{Acknowledgments}
Much of the research in this paper was carried out during a visit by the fourth and fifth authors to the University of Illinois at Urbana-Champaign. The authors are grateful to the BRIDGE strategic alliance between the University of Birmingham and the University of Illinois at Urbana-Champaign, which partially funded this visit. We thank Andrzej Dudek for pointing us towards Lemma~\ref{lem::longpath}, and to the referee for their careful review.

{\noindent \bf Data availability statement.}
There are no additional data beyond that contained within the main manuscript.


\begin{thebibliography}{99}

\bibitem{conf} I. Araujo, József Balogh, Robert A. Krueger, S. Piga, and A. Treglown, Cycles of every length and orientation in randomly perturbed digraphs, \emph{Proceedings of the 12th European Conference on Combinatorics, Graph Theory and Applications} (2023), 66--73. 

\bibitem{BKS} I. Ben-Eliezer, M. Krivelevich, and B. Sudakov, The size Ramsey number of a directed path, \emph{J. Combin. Theory Ser. B}~{\bf 102} (2012), 743--755.

\bibitem{bfm1} T. Bohman, A. Frieze, and R. Martin, How many edges make a dense graph Hamiltonian?, \emph{Random Struct. Alg.}~{\bf 22} (2003), 33--42.

\bibitem{bondy2} J.A. Bondy, Pancyclic graphs I, \emph{J. Combin. Theory Ser. B} {\bf 11} (1971), 80--84.

\bibitem{bondy1} J.A. Bondy, Pancyclic graphs: recent results, infinite and finite sets, \emph{Colloq. Math. Soc. J\'anos Bolyai, Keszthely} (1973), 181--187.




\bibitem{bhkmpp} J. B\"ottcher, J. Han, Y. Kohayakawa, R. Montgomery, O. Parczyk, and Y. Person, Universality for bounded degree spanning trees in randomly perturbed graphs, \emph{Random Struct. Alg.} {\bf 55} (2019), 854--864.




\bibitem{dkmot} L. DeBiasio, D. Kühn, T. Molla, D. Osthus, and A. Taylor, Arbitrary orientations of Hamilton cycles in digraphs,
\emph{SIAM J. Discrete Math.} {\bf 29} (2015), 1553--1584.
\bibitem{deb} L. DeBiasio and T. Molla, Semi-degree threshold for anti-directed Hamilton cycles,
\emph{Electron. J. Combin.} {\bf 22} (2015), P4.34.
\bibitem{D} G.A. Dirac, Some theorems on abstract graphs,
{\em Proc. Lond. Math. Soc.} {\bf 2} (1952), 69--81.

\bibitem{univ} A. Ferber, G. Kronenberg, and K. Luh, 
Optimal threshold for a random graph to be 2-universal,
\emph{Trans. Am. Math. Soc.}
{\bf 372} (2019), 4239--4262.
\bibitem{ferberkwan} A. Ferber and M. Kwan, Almost all Steiner triple systems are almost resolvable, \emph{Forum Math. Sigma} {\bf 8} (2020), e39.
 \bibitem{fl} A. Ferber and E. Long, Packing and counting arbitrary Hamilton cycles in random digraphs, \emph{Random Struct. Alg.} {\bf 54} (2019), 499--514.

 % \bibitem{Frieze} A.M. Frieze, On large matchings and cycles in sparse random graphs, \emph{Discr. Math.} {\bf 59} (1986), 243--256.
 
 % \bibitem{fpp} A. Frieze, X. P\'erez-Gim\'enez, and P. Prałat, On the existence of Hamilton cycles with a periodic
%pattern in a random digraph,
%\emph{Electron. J. Combin.} {\bf 27} (2020), P4.30.


 \bibitem{GhouilaHouri} A.~Ghouila-Houri, Une condition suffisante d'existence d'un circuit
hamiltonien, \emph{C.R.~Acad.~Sci.~Paris}~\textbf{25} (1960), 495--497.

%\bibitem{HaggkvistHamilton} R.~H\"aggkvist, Hamilton cycles in oriented graphs,
%\emph{Combin. Probab. Comput.}~\textbf{2} (1993), 25--32.

\bibitem{Grant} D. Grant, Antidirected Hamilton cycles in digraphs, \emph{Ars Combinatoria} {\bf 10} (1980), 205--209.
 \bibitem{Hagg} R.~H{\"a}ggkvist and A.~Thomason, Oriented Hamilton cycles in digraphs, \emph{J. Graph Theory} {\bf 20} (1995),
471--479.

%\bibitem{HaggkvistThomasonHamilton} R.~H{\"a}ggkvist and A.~Thomason,
%Oriented Hamilton cycles in oriented graphs,
%in \emph{Combinatorics, Geometry and Probability},
%Cambridge University Press 1997, 339--353.
\bibitem{hmmmp} M. Hahn-Klimroth, G.S. Maesaka, Y. Mogge, S. Mohr, and O. Parczyk,
Random perturbation of sparse graphs,
\emph{Electron. J. Combin.} {\bf 28} (2021),
P2.26.

\bibitem{hmt} J. Han, P. Morris, and A. Treglown, Tilings in randomly perturbed graphs: bridging the gap between Hajnal--Szemer\'edi and Johansson--Kahn--Vu,
\emph{Random Struct. Alg.} {\bf 58} (2021), 480--516.




 \bibitem{joos2} F. Joos and J. Kim, Spanning trees in randomly perturbed graphs, \emph{Random Struct. Alg.} {\bf 56} (2020), 169--219.
 \bibitem{kriv} M. Krivelevich, Triangle factors in random graphs, \emph{Combin. Probab. Comput.} {\bf 6} (1997), 337--347.
 
 \bibitem{kks2} M. Krivelevich, M. Kwan, and B. Sudakov, Cycles and matchings in randomly perturbed digraphs and hypergraphs, \emph{Combin. Probab. Comput.} {\bf 25} (2016), 909--927.

\bibitem{kks1} M. Krivelevich, M. Kwan, and B. Sudakov, Bounded-degree spanning trees in randomly perturbed graphs, \emph{SIAM J. Discrete Math.} {\bf 31} (2017), 155--171.




\bibitem{McDia} C. McDiarmid, General first-passage percolation, \emph{Adv. Appl. Probab.}~ {\bf 15} (1983), 149--161.

\bibitem{monty} R. Montgomery, Spanning trees in random graphs, \emph{Adv. Math.} {\bf 356} (2019), 106793.
\bibitem{mont_embedding} R. Montgomery, Embedding bounded degree spanning trees in random graphs, arXiv:1405.6559.

\bibitem{mont} R. Montgomery, Spanning cycles in random directed graphs, arXiv:2103.06751.


\bibitem{mp} P. Morawski and K. Petrova, Randomly perturbed digraphs also have bounded-degree
spanning trees, arXiv:2306.14648.


%\bibitem{nena2} R. Nenadov and Y. Pehova, On a Ramsey--Turán Variant of the Hajnal--Szemerédi Theorem, \emph{SIAM J. Discr. Math.}~ {\bf 34} (2020), 1001--1010.
 
\bibitem{par} O. Parczyk,
2-universality in randomly perturbed graphs, \emph{European J. Combin.} {\bf 87} (2020), 103--118. 




\bibitem{RRS} V. R\"odl, A. Ruci\'nski, and E. Szemer\'edi, A Dirac-Type Theorem for $3$-Uniform Hypergraphs, \emph{Combin. Probab. Comput.}, {\bf 15} (2006), 229--251.


\end{thebibliography}
\end{document}